\documentclass[11pt]{article}

\usepackage[T1]{fontenc}
\usepackage[utf8]{inputenc}
\usepackage{lmodern}
\usepackage[margin=2cm]{geometry}
\usepackage{microtype}
\usepackage{amsmath,amssymb,mathtools,mathrsfs}
\usepackage{amsthm}
\usepackage{enumitem}
\usepackage{booktabs}
\usepackage{aliascnt}
\usepackage[colorlinks=true,linkcolor=blue,citecolor=blue,urlcolor=blue]{hyperref}
\usepackage[nameinlink,capitalise,noabbrev]{cleveref}

\makeatletter
\AddToHook{cmd/@thm/before}{\LinkTargetOff}
\AddToHook{cmd/@begintheorem/before}{\LinkTargetOn}
\let\dimfree@thmhead\thmhead
\renewcommand{\thmhead}[3]{%
  \if\relax\detokenize{#2}\relax\else
    \MakeLinkTarget{\@currentcounter}%
  \fi
  \dimfree@thmhead{#1}{#2}{#3}%
}
\makeatother

\numberwithin{equation}{section}
\newcommand{\subjclass}[2][]{%
  \par\smallskip\noindent
  \textit{\if\relax\detokenize{#1}\relax Mathematics Subject Classification\else #1 Mathematics Subject Classification\fi.} #2\par}
\newcommand{\keywords}[1]{%
  \noindent\textit{Keywords.} #1\par\medskip}

\let\articletitle\title
\renewcommand{\title}[2][]{\articletitle{#2}}

\newtheorem{theorem}{Theorem}[section]
\newaliascnt{proposition}{theorem}
\newtheorem{proposition}[proposition]{Proposition}
\aliascntresetthe{proposition}
\newaliascnt{lemma}{theorem}
\newtheorem{lemma}[lemma]{Lemma}
\aliascntresetthe{lemma}
\newaliascnt{corollary}{theorem}
\newtheorem{corollary}[corollary]{Corollary}
\aliascntresetthe{corollary}
\newaliascnt{remark}{theorem}
\newtheorem{remark}[remark]{Remark}
\aliascntresetthe{remark}
\newaliascnt{question}{theorem}
\newtheorem{question}[question]{Question}
\aliascntresetthe{question}

\crefname{theorem}{Theorem}{Theorems}
\crefname{proposition}{Proposition}{Propositions}
\crefname{lemma}{Lemma}{Lemmas}
\crefname{corollary}{Corollary}{Corollaries}
\crefname{remark}{Remark}{Remarks}
\crefname{question}{Question}{Questions}

\newcommand{\R}{\mathbb{R}}

\newcommand{\Id}{\mathrm{Id}}
\newcommand{\Lip}{\operatorname{Lip}}
\newcommand{\diam}{\operatorname{diam}}
\newcommand{\supp}{\operatorname{supp}}
\newcommand{\tr}{\operatorname{tr}}
\newcommand{\ran}{\operatorname{ran}}
\newcommand{\op}{\mathrm{op}}

\newcommand{\1}{\mathbf{1}}
\newcommand{\Sd}{\mathbb{S}^{d-1}}
\newcommand{\Ptwo}{\mathcal{P}_{2}}
\newcommand{\dd}{\,\mathrm{d}}
\newcommand{\eps}{\varepsilon}
\newcommand{\cMA}{c_{\mathrm{MA}}}
\newcommand{\Qp}{\mathcal{Q}}
\newcommand{\Sp}{\mathcal{S}}
\newcommand{\Bp}{\mathcal{B}}

\title[Dimension-free Brenier bounds]{Dimension-Free Lipschitz Bounds for Brenier Maps to Compact Log-Concave Targets}
\author{Maja Gw\'o\'zd\'z\\ETH Z\"urich\footnote{Rämistrasse 101, 8092, Zurich, Switzerland}\\\texttt{mgwozdz@ethz.ch}}
\date{}

\begin{document}

\maketitle

\begin{abstract}
Let $d\geq1$, $\Lambda>0$, and $\rho\in\R$, let $\mu=\mathcal{N}(0,\Lambda^{-1}\Id)$ on $\R^{d}$, where $\Id$ is the identity matrix, and let $\nu$ be a probability measure whose support is a convex body $K$ of diameter $D=\diam(K)$ and whose density is proportional to $e^{-W}$, where $\mathrm{D}^{2}W\succeq-\rho\Id$. We prove that the Brenier map $T$ from $\mu$ to $\nu$ has Lipschitz constant $\Lip(T)\leq\sqrt{\Lambda/(2\pi)}\int_{-D/2}^{D/2}e^{\rho t^{2}/2}\dd t$, which is sharp in every fixed dimension. For log-concave targets, the optimal coefficient of $\sqrt{\Lambda}\,D$ is $1/\sqrt{2\pi}$.

For sources with finite second moment and density proportional to $e^{-V}$ for which $x\mapsto\frac12\langle Qx,x\rangle-V(x)$ is convex, where $Q$ is symmetric positive definite, and log-concave targets with compact support $K$, the Brenier potential $\Phi$ satisfies $\partial_{vv}\Phi\leq C_{1}\sqrt{\langle Qv,v\rangle}\,w_{K}(v)$ in the sense of distributions, where $w_{K}(v):=\sup_{y\in K}\langle y,v\rangle-\inf_{y\in K}\langle y,v\rangle$ is the width of $K$ along $v\in\R^{d}$ and $C_{1}<0.463$. This removes Kolesnikov's $\sqrt d$ loss. For log-concave sources with $\mathrm{D}^{2}V\preceq\Lambda\Id$ and the target $\nu$, the Brenier map $T$ satisfies $\Lip(T)^{2}\leq\Lambda\int_{0}^{D^{2}/4}e^{\rho t}\dd t$, so $\Lip(T)$ has the optimal exponential rate $1/8$ in $\rho D^{2}$.
\end{abstract}


\section{Introduction}

\subsection{The problem}

We denote by $T=\nabla\Phi$ the Brenier map for the quadratic cost \cite{Brenier1991,McCann1995} from the standard Gaussian measure $\gamma_{d}$ on $\R^{d}$ to normalised Lebesgue measure on a convex body $K$. Whenever such a map has an everywhere-defined continuous representative, we use the same notation for that representative, and every global Lipschitz constant $\Lip(T)$ follows this convention.

Kolesnikov proved in \cite{Kolesnikov2010} (see also \cite[Theorem 4.2]{Kolesnikov2011}) that there exists a universal numerical constant $C$ such that
\[
\Lip(T)\leq C\sqrt d\,\diam(K).
\]
He then asked whether the factor $\sqrt d$ can be removed \cite[Problem 4.3]{Kolesnikov2011}. The same question appears in \cite[Section 1.1]{MikulincerShenfeld2024} and \cite[Section 1]{MikulincerShenfeld2023}, where the analogous problem was solved for other transport maps.

One could also interpret this question as the endpoint with zero-target curvature of Caffarelli's contraction principle \cite{Caffarelli2000,Caffarelli2002}. At this endpoint, curvature estimates do not give a positive coercivity scale. Indeed, the estimate $\Lip(T)\leq\sqrt{\Lambda/\kappa}$ for a $\Lambda$-semiconcave source potential and a $\kappa$-convex target potential has no content when $\kappa=0$. This implies that the missing coercivity has to come from the target support. Kolesnikov's support argument provides it, but ``at the price'' of a trace loss that depends on the dimension.

In this paper, we remove this loss and identify the sharp Gaussian diameter coefficient. More precisely, for the standard Gaussian source and every compactly supported log-concave target with support $K$, it holds that
\[
\Lip(T)\leq\frac{\diam(K)}{\sqrt{2\pi}},
\]
where the coefficient is optimal in every dimension. Note that our result is not restricted to the Gaussian case. Indeed, we prove a directional affine-covariant estimate for the full source class with an upper bound on the Hessian $\mathrm{D}^{2}V$ of the source potential $V$, in which the source curvature in a direction $v$ is paired with the target width in the same direction. We define the constant $C_{1}$ in \eqref{eq:C1}, and discuss the related interesting questions for directional coefficients in Section \ref{sec:constants}.

Our proof relies on a joint source--translation maximum principle and a Schur complement estimate for the block Hessian, which replace the trace estimate that causes Kolesnikov's $\sqrt d$ loss by a scalar dimension-free coercivity inequality. We describe the mechanism in Section \ref{sec:proof-strategy}.

\subsection{Notation}
Let us settle the notation first. In this work, $d\geq1$ is an integer, $|\cdot|$ is the Euclidean norm on $\R^{d}$, $\Sd$ is the unit sphere, $\Id$ is the identity matrix, $\operatorname{id}_{\R^{d}}$ is the identity map on $\R^{d}$, $\preceq$ is the Loewner order on symmetric matrices, $R^{\alpha}$ is the spectral power of a symmetric matrix $R\succeq0$ (with $R\succ0$ if $\alpha<0$), and $\|\cdot\|_{\op}$ is the operator norm. We write $\Ptwo(\R^{d})$ for the set of Borel probability measures with finite second moment, $W_{2}$ for the quadratic Wasserstein distance on $\Ptwo(\R^{d})$, and $\gamma_{d}$ for the standard Gaussian measure on $\R^{d}$. For a set $E$, $\diam(E)$ denotes its diameter and $\partial E$ its boundary, and we write $B_{r}(y)$ for the open ball of radius $r$ centred at $y$. If $K$ is compact, we define its width in the direction $v$ by
\[
w_{K}(v):=\sup_{y\in K}\langle y,v\rangle-\inf_{y\in K}\langle y,v\rangle,
\qquad v\in\R^{d}.
\]
For a function $f\in C^{2}$ and $v,w\in\R^{d}$, we set $f_{v}:=\langle\nabla f,v\rangle$ and $f_{vw}:=\mathrm{D}^{2}f[v,w]$. If $\Phi$ is a finite convex function on $\R^{d}$, then $\partial\Phi$ denotes its subdifferential, and we write
\[
\partial\Phi(A):=\bigcup_{x\in A}\partial\Phi(x)
\]
for $A\subset\R^{d}$. We also use $\partial_{vv}\Phi$ for the second distributional derivative in the direction $v$. For $c\in\R$, the distributional inequality $\partial_{vv}\Phi\leq c$ is, in equivalent terms, the second-difference estimate
\begin{equation}\label{eq:second-difference}
\Phi(x+hv)+\Phi(x-hv)-2\Phi(x)\leq ch^{2}
\qquad (x\in\R^{d},\ h\in\R).
\end{equation}
We recall this fact in Section \ref{sec:stability}. Whenever a probability density is written as $Z^{-1}f$, the symbol $Z$ denotes $\int f$.

\subsection{Main results}

Our two most important results are the directional support--curvature estimate in Theorem \ref{thm:support-curvature} and the sharp Gaussian diameter theorem in Theorem \ref{thm:sharp-gaussian}, which identifies the exact coefficient for every real target curvature parameter.

\begin{theorem}[Support--curvature theorem]\label{thm:support-curvature}
Let $Q\in\R^{d\times d}$ be symmetric positive definite, and let
\[
\mu:=Z_{\mu}^{-1}e^{-V}\dd x\in\Ptwo(\R^{d}).
\]
We assume that $V:\R^{d}\to\R$ is finite and that
\[
x\longmapsto\frac12\langle Qx,x\rangle-V(x)
\]
is convex. Let $\nu$ be a compactly supported log-concave probability measure on $\R^{d}$ with support $K$, and let $\Phi:\R^{d}\to\R$ be convex with $(\nabla\Phi)_{\#}\mu=\nu$. It follows that, for every $v\in\R^{d}$,
\begin{equation}\label{eq:support-curvature}
\partial_{vv}\Phi\leq C_{1}\sqrt{\langle Qv,v\rangle}\,w_{K}(v)
\end{equation}
in the sense of distributions, where $C_{1}<0.463$ is defined in \eqref{eq:C1}. Moreover, $\nabla\Phi$ has an everywhere-defined globally Lipschitz representative, such that $\nabla\Phi(\R^{d})\subset K$ and
\begin{equation}\label{eq:global-support-curvature}
\Lip(\nabla\Phi)\leq C_{1}\sqrt{\|Q\|_{\op}}\,\diam(K).
\end{equation}
\end{theorem}

The hypothesis on $V$ means, in other words, that the density of $\mu$ is log-convex relative to the centred Gaussian with precision matrix $Q$. The target $\nu$ is an arbitrary compactly supported log-concave probability measure in the sense of Borell \cite{Borell1975}, so $K$ may have empty interior. In this case, the right-hand side of \eqref{eq:support-curvature} vanishes whenever $v\perp\operatorname{span}(K-K)$. Because the Brenier potential is unique up to an additive constant, the estimate does not depend on its normalisation. Both sides of \eqref{eq:support-curvature} transform in the same way under affine changes of variables. Finally, the bound \eqref{eq:global-support-curvature} follows if we take $v$ of unit length. For the standard Gaussian, we have $V(x)=|x|^{2}/2$ and $Q=\Id$.

\begin{corollary}[Gaussian source]\label{cor:gaussian-source}
Let $\nu$ be a compactly supported log-concave probability measure with support $K$, and let us denote by $T=\nabla\Phi$ the Brenier map from $\gamma_{d}$ to $\nu$. It follows that
\[
\partial_{ee}\Phi\leq C_{1}w_{K}(e)
\qquad (e\in\Sd)
\]
in the sense of distributions. Moreover, $T$ has an everywhere-defined globally Lipschitz representative that satisfies
\[
\Lip(T)\leq\frac{\diam(K)}{\sqrt{2\pi}}.
\]
The coefficient in the diameter estimate is optimal, and the same statements hold for targets supported on proper affine subspaces.
\end{corollary}

We prove Corollary \ref{cor:gaussian-source}, which includes the case of singular targets, in Section \ref{sec:sharp-diameter}.

The coercivity that comes from compact support persists under bounded negative target curvature. To make this precise, let $K\subset\R^{d}$ be a convex body, set $\Omega:=\operatorname{int}K$, and let us consider
\[
\nu:=Z_{\nu}^{-1}e^{-W}\1_{\Omega}\dd y,
\qquad Z_{\nu}:=\int_{\Omega}e^{-W}\dd y\in(0,\infty),
\]
where $W:\Omega\to\R$. For $\rho\in\R$, we impose convexity of
\[
y\longmapsto W(y)+\frac{\rho}{2}|y|^{2}
\qquad\text{on }\Omega.
\]
If $\rho\geq0$, we call $\nu$ $\rho$-semi-log-concave. If $\rho=-\kappa$, this is the usual $\kappa$-log-concavity condition.

\begin{theorem}[Semi-log-concave targets]\label{thm:semilogconcave}
Let $\rho\geq0$ and $\Lambda>0$, and let
\[
\mu:=Z_{\mu}^{-1}e^{-V}\dd x\in\Ptwo(\R^{d}).
\]
We assume that $V:\R^{d}\to\R$ is finite and that
\[
x\longmapsto\frac{\Lambda}{2}|x|^{2}-V(x)
\]
is convex. Let $\nu$ be a $\rho$-semi-log-concave probability measure whose support is a convex body $K$, and let $T=\nabla\Phi$ be the Brenier map from $\mu$ to $\nu$. We set
\[
r:=\rho\diam(K)^{2},
\]
and let $C_{2}(r)$ be the constant from \eqref{eq:C2}. The function $C_{2}$ is nondecreasing, $C_{2}(0)=C_{1}$, and
\[
C_{2}(r)\leq e^{3/4}r^{1/4}e^{r/4}
\qquad (r\geq1).
\]
With these assumptions, $T$ has an everywhere-defined globally Lipschitz representative, and
\[
\Lip(T)\leq C_{2}(r)\sqrt{\Lambda}\,\diam(K).
\]
\end{theorem}

For $\rho=0$, we recover the operator-norm part of Theorem \ref{thm:support-curvature}, while for $\rho>0$, the theorem gives the exponential rate $1/4$.

In our second main result, we use an inverse two-point argument to find the exact Gaussian diameter coefficient at every curvature. To this end, we define
\begin{equation}\label{eq:B-def}
\Bp(r):=\frac{1}{\sqrt{2\pi}}\int_{-1/2}^{1/2}e^{rt^{2}/2}\dd t,
\qquad r\in\R.
\end{equation}

\begin{theorem}[Sharp Gaussian diameter bound]\label{thm:sharp-gaussian}
Let $\Lambda>0$ and $\rho\in\R$, and let $\mu=\mathcal{N}(0,\Lambda^{-1}\Id)$ on $\R^{d}$. Let $K\subset\R^{d}$ be a convex body of diameter $D>0$, set $\Omega:=\operatorname{int}K$, and suppose $W:\Omega\to\R$ satisfies
\[
Z_{\nu}:=\int_{\Omega}e^{-W}\dd y\in(0,\infty),
\qquad W+\frac{\rho}{2}|\cdot|^{2}\text{ is convex on }\Omega.
\]
It follows that the Brenier map $T$ from $\mu$ to $\nu:=Z_{\nu}^{-1}e^{-W}\1_{\Omega}\dd y$ has an everywhere-defined globally Lipschitz representative and
\begin{equation}\label{eq:sharp-gaussian}
\Lip(T)\leq\sqrt{\Lambda}\,D\Bp(\rho D^{2})
=\sqrt{\frac{\Lambda}{2\pi}}\int_{-D/2}^{D/2}e^{\rho t^{2}/2}\dd t.
\end{equation}
Moreover, for every fixed $d\geq1$, $\Lambda>0$, $D>0$, and $\rho\in\R$, the right-hand side is the exact supremum over the class of full-dimensional targets that satisfy these assumptions. In dimension one, equality is attained by the density proportional to $e^{\rho y^{2}/2}$ on $[-D/2,D/2]$, whereas in higher dimensions, full-dimensional product targets of diameter $D$ approach the same value.
\end{theorem}

In particular, we have $\Bp(0)=1/\sqrt{2\pi}$. For $\rho<0$, the bound \eqref{eq:sharp-gaussian} combines positive target curvature with finite diameter. The same two-point method also applies to convex, not necessarily quadratic, source potentials.

\begin{theorem}[Convex sources and the sharp exponential rate]\label{thm:convex-source}
Let $\Lambda>0$ and $\rho\in\R$. We assume that $V:\R^{d}\to\R$ and $\frac{\Lambda}{2}|\cdot|^{2}-V$ are convex, and let $\mu:=Z_{\mu}^{-1}e^{-V}\dd x$ be a probability measure. Let $\nu$, $K$, and $D$ satisfy the assumptions of Theorem \ref{thm:sharp-gaussian}. It follows that the Brenier map $T$ from $\mu$ to $\nu$ has an everywhere-defined globally Lipschitz representative and
\begin{equation}\label{eq:convex-source}
\Lip(T)^{2}\leq\Lambda\int_{0}^{D^{2}/4}e^{\rho t}\dd t
=
\begin{cases}
\frac{\Lambda}{\rho}\bigl(e^{\rho D^{2}/4}-1\bigr),&\rho\neq0,\\[1ex]
\frac{\Lambda D^{2}}{4},&\rho=0.
\end{cases}
\end{equation}
For $\rho\geq0$, the optimal exponential rate in $r=\rho D^{2}$ for this source class is $1/8$, even in each fixed dimension and for full-dimensional targets.
\end{theorem}

In dimension one, we can extend the exact Gaussian coefficient to the full source class by monotone composition.

\begin{theorem}[Exact one-dimensional coefficient]\label{thm:one-dimensional}
Let $\Lambda>0$ and $\rho\in\R$, and let
\[
\mu:=Z_{\mu}^{-1}e^{-V}\dd x\in\Ptwo(\R),
\]
where $V:\R\to\R$ is finite and $x\mapsto\frac{\Lambda}{2}x^{2}-V(x)$ is convex. Let $a<b$, set $D:=b-a$, and suppose that $W:(a,b)\to\R$ is finite with
\[
Z_{\nu}:=\int_{a}^{b}e^{-W}\dd y\in(0,\infty),
\qquad W+\frac{\rho}{2}|\cdot|^{2}\text{ convex on }(a,b).
\]
It follows that the monotone Brenier map $T$ from $\mu$ to $\nu:=Z_{\nu}^{-1}e^{-W}\1_{(a,b)}\dd y$ satisfies
\begin{equation}\label{eq:one-dimensional}
\Lip(T)\leq\sqrt{\Lambda}\,D\Bp(\rho D^{2})
=\sqrt{\frac{\Lambda}{2\pi}}\int_{-D/2}^{D/2}e^{\rho t^{2}/2}\dd t.
\end{equation}
The coefficient is optimal for every $\Lambda$, $\rho$, and $D$. Equality holds for a Gaussian source of variance $\Lambda^{-1}$ and the target density proportional to $e^{\rho(y-(a+b)/2)^{2}/2}$ on $(a,b)$.
\end{theorem}

For log-concave targets, we obtain from Theorem \ref{thm:one-dimensional} that
\begin{equation}\label{eq:one-dimensional-logconcave}
\Lip(T)\leq\frac{\sqrt{\Lambda}D}{\sqrt{2\pi}},
\end{equation}
with equality for a Gaussian source and the uniform target. As $r\to\infty$, it holds that
\[
\Bp(r)\sim\frac{4}{\sqrt{2\pi}\,r}e^{r/8}.
\]
In particular, both the exponential rate $1/8$ and the power $r^{-1}$ are optimal in the one-dimensional coefficient.

\subsection{Constants}\label{sec:constants}

Let $C_{*}^{\mathrm{full}}$ denote the least constant for which \eqref{eq:support-curvature} holds with $C_{*}^{\mathrm{full}}$ in place of $C_{1}$ in every dimension and for all data in Theorem \ref{thm:support-curvature}. We also let $C_{*}^{G}$ be the least constant such that
\[
\partial_{ee}\Phi\leq C_{*}^{G}w_{K}(e)
\qquad (e\in\Sd)
\]
holds in every dimension for the standard Gaussian source and every compactly supported log-concave target. For the uniform target on $[0,\ell]\times[0,1]^{d-1}$, where $\ell>0$, and the first standard basis vector $e_{1}$, the product Brenier potential satisfies
\[
\Phi_{e_{1}e_{1}}(x)=\ell(2\pi)^{-1/2}\exp(-x_{1}^{2}/2).
\]
This gives the following chain of inequalities:
\[
\frac{1}{\sqrt{2\pi}}\leq C_{*}^{G}\leq C_{*}^{\mathrm{full}}\leq C_{1}.
\]
Both directional constants equal $1/\sqrt{2\pi}$ if we restrict them to dimension one, but their values over all dimensions are open (see Question \ref{q:open}). The sharp Gaussian diameter theorem does not solve these questions, since $w_{K}(e)$ may be much smaller than $\diam(K)$. In Remark \ref{rem:scalar-optimality}, we explain that $C_{1}$ is the optimal value of the particular scalar criterion from Proposition \ref{prop:profile}.

For the semi-log-concave problem, let us consider estimates of the form
\[
\Lip(T)\leq C(1+r)^{N}e^{cr}\sqrt{\Lambda}\,\diam(K),
\qquad r=\rho\diam(K)^{2}\geq0,
\]
where $C>0$ and $N\geq0$ are independent of the dimension and of the admissible data. Let $c_{\mathrm{sl}}^{*}$ be the infimum of the possible exponents $c$ for the source class of Theorem \ref{thm:semilogconcave}, and we write $c_{\mathrm{sl}}^{G}$ for isotropic Gaussian sources and $c_{\mathrm{sl}}^{\mathrm{cvx}}$ for the convex sources of Theorem \ref{thm:convex-source}. Our results show that
\[
c_{\mathrm{sl}}^{G}=c_{\mathrm{sl}}^{\mathrm{cvx}}=\frac18,
\qquad \frac18\leq c_{\mathrm{sl}}^{*}\leq\frac14.
\]
Note that the lower bounds already follow from Gaussian product examples in every fixed dimension.

\subsection{Related works}\label{sec:previous-work}

\paragraph{Curvature.}
The support--curvature theorem complements Caffarelli's contraction principle. In its version for general sources, the latter states that, if $\mathrm{D}^{2}V\preceq\Lambda\Id$ and $\mathrm{D}^{2}W\succeq\kappa\Id$ for some $\Lambda,\kappa>0$, then the Brenier map from $e^{-V}\dd x$ to $e^{-W}\dd y$ is $\sqrt{\Lambda/\kappa}$-Lipschitz \cite{Caffarelli2000,Caffarelli2002}. This general statement appears explicitly in Theorem 1 of Fathi, Gozlan, and Prod'homme (see \cite{ColomboFigalliJhaveri2017,FathiGozlanProdhomme2020,ChewiPooladian2023,KimMilman2012} for this version, further proofs, and extensions). When $\kappa=0$, this estimate has no coercive scale. In Theorem \ref{thm:support-curvature}, we instead derive coercivity from compact support and pair it directionally with the source curvature. This shows that positive target curvature and bounded target support provide two distinct mechanisms for global Lipschitz control, and when both are present, one may take the smaller of the two bounds.

Klartag and Kolesnikov obtained dimension-free distributional information on the Hessian spectrum \cite{KlartagKolesnikov2015}. More precisely, under the assumption that the Brenier potential is $C^{2}$ on the interior of $\supp(\mu)$, they proved multiplicative-scale concentration estimates for the eigenvalues of $\mathrm{D}^{2}\Phi$ between absolutely continuous log-concave measures. These estimates control fluctuations of the eigenvalues, but not their pointwise scale or an $L^{\infty}$ operator-norm bound. In particular, they do not imply the support-controlled estimate that we prove here.

Recent global growth and regularity estimates for Brenier maps rely on target convexity moduli or on further quantitative assumptions on the source and target densities or potentials \cite{Kolesnikov2014,ColomboFathi2021,Fathi2024,DePhilippisShenfeld2025,CarlierFigalliSantambrogio2026,GozlanSylvestre2025,Bidoia2026,AmmariFigalli2026}. Among the closest results, Gozlan and Sylvestre derive global directional and anisotropic modulus estimates from quantitative moduli of smoothness of the source potential and of convexity of the target potential \cite[Section 4]{GozlanSylvestre2025}, and their conclusions also cover measures supported on affine subspaces. For normalised Lebesgue measure on a convex body, however, the target convexity modulus vanishes at every scale below the support diameter. In this case, their estimate gives an oscillation bound, but not a local Lipschitz estimate. We note that none of the assumptions that give Lipschitz control in \cite{GozlanSylvestre2025,Bidoia2026,AmmariFigalli2026} follows merely from compactness of the target support.

\paragraph{Heat flow and Brownian transport maps.}
A complementary line of work constructs dimension-free Lipschitz transports which need not be optimal, and support-controlled bounds are known for such maps from Gaussian measure. For the Kim--Milman heat flow map $\phi^{\mathrm{flow}}$ \cite{KimMilman2012}, Mikulincer and Shenfeld prove in version 5 of \cite[Theorem 1]{MikulincerShenfeld2023}, dated 17 July 2026, that, if $\nu$ is $\kappa$-log-concave for some $\kappa\in\R$, $D:=\diam(\supp\nu)$, and $\kappa D^{2}<4$, then
\[
\|\nabla\phi^{\mathrm{flow}}\|_{\op}
\leq\frac{D}{2}\exp\!\left(\frac{1-\kappa D^{2}/4}{2}\right).
\]
For their Brownian transport map from Wiener space, version 4 of \cite[Theorem 1.1]{MikulincerShenfeld2024}, dated 27 August 2026, gives, in the same regime, the almost-sure contraction constant
\[
D\left(\frac{e^{2-\kappa D^{2}/2}+1}{8}\right)^{1/2}.
\]
An order-$D$ estimate at $\kappa=0$ was therefore already available for these transports. Their stated coefficients at $\kappa=0$ are $\sqrt e/2$ and $\sqrt{(e^{2}+1)/8}$, whereas Theorem \ref{thm:sharp-gaussian} gives the exact coefficient $1/\sqrt{2\pi}$ for the Brenier map. When $\kappa=-\rho<0$, the heat flow estimate has exponential rate $1/8$ in $\rho D^{2}$, and the Brownian estimate has rate $1/4$. For the Brenier map, the sharp Gaussian theorem gives rate $1/8$. Theorem \ref{thm:convex-source} extends this rate to all convex source potentials with the prescribed upper Hessian bound, while the rate $1/4$ in Theorem \ref{thm:semilogconcave} applies to the larger class which allows nonconvex source potentials.

Further dimension-free Lipschitz bounds for heat flow maps are due to Neeman \cite[Theorem 1.3]{Neeman2022} for bounded perturbations of the Gaussian whose perturbing potential has a uniform lower Hessian bound, Klartag and Putterman \cite{KlartagPutterman2023}, Fathi, Mikulincer, and Shenfeld \cite{FathiMikulincerShenfeld2024}, Brigati and Pedrotti \cite{BrigatiPedrotti2025}, Conforti and Eichinger \cite{ConfortiEichinger2025}, and, under dimension-independent control of the relevant weak convexity profile, Chaintron, Conforti, and Eichinger \cite{ChaintronConfortiEichinger2025}. For stability of the Kim--Milman map, consult \cite{ChewiEichingerPooladian2026}. Both papers of Mikulincer and Shenfeld identify the corresponding Brenier problem as open, and the case $\kappa=0$ is exactly Kolesnikov's question. Note that, for non-isotropic Gaussian sources, the heat flow and Brenier maps need not coincide \cite[Proposition 4.1]{Tanana2021}. Moreover, for the endpoint-dependent target measures considered by Mikulincer and Shenfeld, the optimal transport map on Wiener space essentially reduces to the finite-dimensional Brenier map \cite[Section 8]{MikulincerShenfeld2024}, so the Brenier estimates do not follow from the Wiener-space results. In dimension one, the heat flow map from the Gaussian is generated by a vector field on $\R$, so it is increasing. This implies that it coincides with the monotone rearrangement, and so the exact coefficient in Theorems \ref{thm:sharp-gaussian} and \ref{thm:one-dimensional} applies to this map as well.

\paragraph{One-dimensional profiles and two-point methods.}
For the Gaussian source and log-concave targets in dimension one, the coefficient $1/\sqrt{2\pi}$ can already be derived from the isoperimetric literature. Montenegro proved a sharp isoperimetric inequality for log-concave probability measures on a convex body $K$ in terms of $\diam(K)$, and he noted, without proof, that it implies a comparison with the Gaussian isoperimetric profile with the factor $\sqrt{2\pi}/\diam(K)$ \cite[(2) and (9)]{Montenegro2006}. Applied to half-lines, this comparison is the inequality of Proposition \ref{prop:target-profile-equality}, so it gives $\Lip(T)\leq D/\sqrt{2\pi}$ for the increasing map $T$ from $\gamma_{1}$ to a log-concave probability measure on an interval of length $D$. In dimension one, the results here extend this bound to all $\rho\in\R$ and to the full source class of Theorem \ref{thm:one-dimensional}, and they identify the equality case in Proposition \ref{prop:target-profile-equality}. The extremal densities, which are proportional to $e^{\rho y^{2}/2}$ on $[-D/2,D/2]$, belong to the one-dimensional model spaces of Milman for the curvature--dimension--diameter condition with infinite dimension and curvature lower bound $-\rho$ \cite[Corollary 1.4]{Milman2015}.

Two-point maximum principles for the Monge--Amp\`ere equation, of which the inverse comparison of Section \ref{sec:sharp-diameter} is an instance, also appear in recent work of Hung, Liu, and Trudinger, who give new proofs of interior second derivative estimates by coupling the inverse Hessians at the two points in a positive semidefinite block matrix \cite{HungLiuTrudinger2026}. Their estimates are interior and depend on the dimension, whereas the bounds here are global and dimension-free.

\subsection{Proof strategy}\label{sec:proof-strategy}

Let us begin with the smooth full-dimensional source and target densities, and temporarily assume a global Hessian bound for the Brenier potential, which we later remove by approximation in Section \ref{sec:approximation}. We fix $e\in\Sd$ and write
\[
A:=\mathrm{D}^{2}\Phi,
\qquad T:=\nabla\Phi,
\qquad \mathcal{L}f:=\tr(A^{-1}\mathrm{D}^{2}f)-\langle\nabla W(T),\nabla f\rangle.
\]
Compact support occurs because of the translated Bregman corrector
\[
G_{se}(x):=\Phi(x+se)-\Phi(x)-s\Phi_{e}(x)
=\int_{0}^{s}(s-r)\Phi_{ee}(x+re)\dd r
\qquad (x\in\R^{d},\ s\in\R).
\]
For related symmetric second-difference arguments, see also \cite{Valdimarsson2007,Kolesnikov2010}. The corrector is a weighted integral of $\Phi_{ee}$ along the segment from $x$ to $x+se$. On the other hand, convexity and the range constraint give
\[
0\leq G_{se}\leq w_{K}(e)|s|.
\]
For small $s$, the corrector behaves, intuitively, like $\Phi_{ee}(x)s^{2}/2$, while it grows at most linearly in $|s|$. This implies that a large directional curvature at one point must be compensated by smaller curvature further along the segment, and the maximum principle below enables it. The linearised Monge--Amp\`ere equation bounds $\mathcal{L}G_{se}$ from below by the Bregman divergence of $-\log\det$ between $\mathrm{D}^{2}\Phi(x+se)$ and $\mathrm{D}^{2}\Phi(x)$, minus $\Lambda s^{2}/2$, where $\mathrm{D}^{2}V\preceq\Lambda\Id$.

The main step is to maximise $G_{se}(x)-\Theta(s)$ jointly in the source point $x$ and the translation $s$, where $\Theta$ is an appropriate penalty. This follows the one-dimensional idea of Kolesnikov \cite[Remark 2.1]{Kolesnikov2010}. Our new contribution is that we estimate the corrector itself rather than $\Phi_{ee}$ and use the full Hessian in $(x,s)$ instead of only its $x$-block.

At a maximum point, stationarity in $s$ makes the support bound active and allows us to express $\Phi_{ee}$ by means of $\Theta'$. The dimension-free gain comes from the second-order information. More precisely, the Schur complement of the $(s,s)$ entry controls the directional curvature at the translated point, and with a spectral Jensen inequality for the Bregman divergence, this reduces the multidimensional maximum principle argument to the scalar coercivity inequality of Lemma \ref{lem:coercivity}. As a result, a maximum at a translation $s$ is impossible unless
\[
\Theta'(s)\leq s\Theta''(s)\Xi^{-1}(\Lambda s^{2}/2),
\qquad \Xi(y):=y+2\log y-\frac1y.
\]
This implies that the admissible penalties must obey a one-dimensional differential inequality. Penalties with $\Theta''(0)$ arbitrarily close to $C_{1}\sqrt{\Lambda}\,w_{K}(e)$ from above give $G_{se}\leq\Theta(s)$. It now suffices to pass to the limiting constant to obtain
\[
\Phi_{ee}\leq C_{1}\sqrt{\Lambda}\,w_{K}(e),
\]
which removes the trace loss in \cite{Kolesnikov2011}. For semi-log-concave targets, we prove Theorem \ref{thm:semilogconcave} in the same way, with a vector translation $m\in\R^{d}$ and the generalised Schur complement with a pseudoinverse.

For the sharp diameter estimates, we use a complementary comparison for the inverse potential $\Psi=\Phi^{*}$, the Legendre transform of $\Phi$. The idea is to compare the component of $\nabla\Psi(y)-\nabla\Psi(x)$ along the target chord $y-x$, for distinct $x,y\in\Omega$, with a one-dimensional profile. This is a two-point comparison in a similar vein to the proof of the fundamental gap conjecture by Andrews and Clutterbuck \cite{AndrewsClutterbuck2011}. At a negative interior minimum of the difference, a reflection coupling adapted to the inverse Hessians gives a second-order inequality. For the Gaussian source, the score is linear, and the comparison closes once we have chosen the profile from the exact interval model. Finally, for convex sources, we replace the score difference by an averaged source Hessian, which gives \eqref{eq:convex-source}.

\subsection{Organisation}

The paper is organised as follows. In Section \ref{sec:smooth-estimates}, we prove the maximum principle estimates in the smooth setting, namely, the profile estimate and the constant $C_{1}$ in Propositions \ref{prop:profile} and \ref{prop:C1}, the semi-log-concave estimate in Proposition \ref{prop:semilog-estimate}, and the properties of $C_{2}$ in Proposition \ref{prop:C2}. In Section \ref{sec:approximation}, we remove the regularity assumptions and prove Theorems \ref{thm:support-curvature} and \ref{thm:semilogconcave}. Section \ref{sec:sharp-diameter} presents the sharp Gaussian diameter theorem and the convex-source estimate, including sharpness of the exponential rate, and Section \ref{sec:one-dimension} gives the exact one-dimensional coefficient for the full source class with the equality case of the pointwise bound for log-concave targets. Finally, in Section \ref{sec:optimality}, we pose some open questions on directional width and nonconvex sources.

\section{Maximum principle estimates in the smooth setting}\label{sec:smooth-estimates}

\subsection{The smooth setting}

In this section, we work with the following assumptions.
\begin{enumerate}[label=(S\arabic*),leftmargin=3.2em]
\item $V\in C^{\infty}(\R^{d})$ and $e^{-V}$ is integrable. The measure $\mu:=Z_{\mu}^{-1}e^{-V}\dd x$ has finite second moment, and $\mathrm{D}^{2}V\preceq\Lambda\Id$ on $\R^{d}$ for some $\Lambda>0$.
\item $\Omega\subset\R^{d}$ is bounded, open, and convex, $K:=\overline\Omega$, and $W\in C^{\infty}(U)$ for an open set $U\supset K$. Moreover,
\[
Z_{\nu}:=\int_{\Omega}e^{-W}\dd y,
\qquad
\nu:=Z_{\nu}^{-1}e^{-W}\1_{\Omega}\dd y.
\]
\item The Brenier potential $\Phi$ from $\mu$ to $\nu$ satisfies
\[
M_{\Phi}:=\sup_{\R^{d}}\|\mathrm{D}^{2}\Phi\|_{\op}<\infty.
\]
\end{enumerate}
We make convexity assumptions on $W$, and also set:
\[
T:=\nabla\Phi,
\qquad A:=\mathrm{D}^{2}\Phi,
\qquad \cMA:=\log Z_{\nu}-\log Z_{\mu}.
\]
We will use this notation whenever \textup{(S1)}--\textup{(S2)} are assumed, with $\Phi$ the Brenier potential from $\mu$ to $\nu$.

\begin{proposition}[Regularity]\label{prop:regularity}
Under \textup{(S1)}--\textup{(S2)}, the map $T:\R^{d}\to\Omega$ is a $C^{\infty}$ diffeomorphism, $A(x)\succ0$ for every $x$, and
\begin{equation}\label{eq:MA}
\log\det A(x)=\cMA-V(x)+W(T(x))
\qquad (x\in\R^{d}).
\end{equation}
\end{proposition}

\begin{proof}
Let us first consider the case $d=1$, in which the map is the monotone rearrangement $T=F_{\nu}^{-1}\circ F_{\mu}$, where $F_{\mu}$ and $F_{\nu}$ are the distribution functions of $\mu$ and $\nu$. Note that the source density is smooth and strictly positive on $\R$, and the target density is smooth and strictly positive on $\Omega$. By the inverse function theorem, $T$ is a smooth increasing diffeomorphism from $\R$ onto $\Omega$.

We now take $d\geq2$. The source density and its reciprocal are bounded on every ball, and the target density and its reciprocal are bounded on $\Omega$. Furthermore, the target domain is convex, and the source domain is all of $\R^{d}$. By \cite[Theorem 1.1, cases (1), (2-a), and (2-b)]{CorderoFigalli2019}, we know that $T:\R^{d}\to\Omega$ is a homeomorphism. Given that both densities are smooth on their open supports, the higher regularity conclusion in the same theorem shows that $T$ is a $C^{\infty}$ diffeomorphism. Finally, \eqref{eq:MA} is the pointwise change-of-variables identity.
\end{proof}

Notice that the translation $y\mapsto y-y_{0}$, with $y_{0}\in\Omega$, replaces $\Phi$ by $\Phi-\langle y_{0},\cdot\rangle$, so a translation of the target changes neither the Hessian nor any directional width. Because of this, we may assume that $0\in\Omega$. We claim that $\Phi$ is then coercive. Indeed, let us choose $r>0$ with $\overline{B_{r}(0)}\subset\Omega$. Since $T=\nabla\Phi$ is a diffeomorphism from $\R^{d}$ onto $\Omega$, the Legendre transform $\Phi^{*}$ is finite on $\Omega$, and the Fenchel inequality gives
\[
\Phi(x)\geq\left\langle x,\frac{rx}{|x|}\right\rangle
-\Phi^{*}\!\left(\frac{rx}{|x|}\right)
\geq r|x|-\sup_{|y|\leq r}\Phi^{*}(y).
\]
This implies that $\Phi(x)\to\infty$ as $|x|\to\infty$. After we add a constant, we may also arrange that
\begin{equation}\label{eq:normalisation}
0\in\Omega,
\qquad \Phi\geq0\quad\text{on }\R^{d}.
\end{equation}

\subsection{The linearised Monge--Amp\`ere operator}

For $f\in C^{2}(\R^{d})$, we define the drifted linearised Monge--Amp\`ere operator by
\begin{equation}\label{eq:L}
\mathcal{L}f:=\tr(A^{-1}\mathrm{D}^{2}f)-\langle\nabla W(T),\nabla f\rangle.
\end{equation}
This is the operator obtained by the linearisation of \eqref{eq:MA}. It plays the role of the elliptic operator in our maximum principles. For a symmetric positive-definite matrix $R$, let us set
\begin{equation}\label{eq:H-g}
\mathcal{H}(R):=\tr R-d-\log\det R\geq0,
\qquad g(q):=q-1-\log q\quad(q>0).
\end{equation}
In particular, $\mathcal{H}(R)=\sum_{i}g(\lambda_{i})$, where $(\lambda_{i})$ are the eigenvalues of $R$, and $\mathcal{H}(R)$ is the Bregman divergence of $-\log\det$ between $R$ and $\Id$ \cite{Bregman1967}. More generally, for a differentiable function $F$ on an open convex set, we write
\[
D_{F}(y,z):=F(y)-F(z)-\langle\nabla F(z),y-z\rangle.
\]
If $F$ is convex, we have $D_{F}\geq0$. Moreover, if $\mathrm{D}^{2}F\preceq\Lambda\Id$, then $D_{F}(y,z)\leq\frac{\Lambda}{2}|y-z|^{2}$, and if $\mathrm{D}^{2}F\succeq-\rho\Id$ for some $\rho\in\R$, then $D_{F}(y,z)\geq-\frac{\rho}{2}|y-z|^{2}$.

\begin{lemma}[The penalty $\Phi$]\label{lem:penalty-Phi}
Assume \textup{(S1)}--\textup{(S2)} and \eqref{eq:normalisation}. If $W$ is convex on $\Omega$, then
\[
\mathcal{L}\Phi\leq C_{W}:=d+W(0)-\inf_{\Omega}W<\infty.
\]
\end{lemma}

\begin{proof}
Observe that $\mathcal{L}\Phi=d-\langle\nabla W(T),T\rangle$. By convexity, \(\langle\nabla W(T),T\rangle\geq W(T)-W(0) \geq\inf_{\Omega}W-W(0)\), which proves the claim.
\end{proof}

We now focus on the target support that is crucial for translations. For $m\in\R^{d}$, we set
\begin{equation}\label{eq:translated-corrector}
\begin{aligned}
G_{m}(x)&:=\Phi(x+m)-\Phi(x)-\langle m,T(x)\rangle,\\
C_{m}(x)&:=A(x+m),\\
R_{m}(x)&:=A(x)^{-1/2}C_{m}(x)A(x)^{-1/2}.
\end{aligned}
\end{equation}

\begin{lemma}[Translated corrector]\label{lem:translated-corrector}
Under \textup{(S1)}--\textup{(S2)}, the following statements hold for each $m,x\in\R^{d}$.
\begin{enumerate}[label=\textup{(\roman*)},leftmargin=2.8em]
\item $0\leq G_{m}(x)\leq w_{K}(m)$. In particular, for every $e\in\Sd$,
\[
G_{se}(x)\leq w_{K}(e)|s|
\qquad (s\in\R).
\]
\item
\[
\nabla_{m}G_{m}=T(x+m)-T(x),
\qquad \mathrm{D}^{2}_{mm}G_{m}=C_{m},
\qquad \mathrm{D}^{2}_{xm}G_{m}=C_{m}-A.
\]
\item
\[
\mathcal{L}G_{m}(x)=\mathcal{H}(R_{m}(x))-D_{V}(x+m,x)
+D_{W}(T(x+m),T(x)).
\]
\end{enumerate}
In particular, if $W$ is convex, then
\[
\mathcal{L}G_{m}\geq\mathcal{H}(R_{m})-\frac{\Lambda}{2}|m|^{2}.
\]
If instead $\mathrm{D}^{2}W\succeq-\rho\Id$ on $\Omega$ for some $\rho\in\R$, then
\[
\mathcal{L}G_{m}\geq\mathcal{H}(R_{m})-\frac{\Lambda}{2}|m|^{2}
-\frac{\rho}{2}|T(x+m)-T(x)|^{2}.
\]
\end{lemma}

\begin{proof}
For the first point, the lower bound is the supporting hyperplane inequality. For the upper bound, we use that $T$ takes values in $K$, which gives \(G_{m}(x)=\int_{0}^{1}\langle T(x+tm)-T(x),m\rangle\dd t \leq w_{K}(m)\). The second point follows by direct differentiation.

We now prove the third point. Since $\nabla[\Phi(\cdot+m)]=T(\cdot+m)$ and $\mathrm{D}^{2}[\Phi(\cdot+m)]=C_{m}$, we get
\[
\mathcal{L}[\Phi(\cdot+m)-\Phi]
=\tr(A^{-1}C_{m})-d
-\langle\nabla W(T(x)),T(x+m)-T(x)\rangle.
\]
Differentiation of \eqref{eq:MA} in the direction $m$ yields
\[
\mathcal{L}\langle m,T\rangle
=\tr(A^{-1}\partial_{m}A)-\langle\nabla W(T),Am\rangle
=-\langle\nabla V,m\rangle.
\]
Finally, \(\tr(A^{-1}C_{m})-d =\mathcal{H}(R_{m})+\log\det C_{m}-\log\det A\), and by \eqref{eq:MA}, \(\log\det C_{m}-\log\det A =-V(x+m)+V(x)+W(T(x+m))-W(T(x))\). It suffices to combine these identities to obtain the formula. Observe that the drift in $\mathcal{L}$ and the linear term in the definition of $G_{m}$ are precisely what turn the differences of $W$ and $V$ into Bregman divergences. The final two bounds then follow from the estimates for the Bregman divergences.
\end{proof}

\subsection{Scalar coercivity}

For $y>0$, we define
\begin{equation}\label{eq:Xi}
\Xi(y):=y+2\log y-\frac1y,
\qquad
\Xi_{+}(y):=\max\{\Xi(y),0\}
=
\begin{cases}
0,&y\leq1,\\
\Xi(y),&y>1.
\end{cases}
\end{equation}
Notice that $\Xi$ is increasing, $\Xi(1)=0$, and $\Xi(y)\geq y-1$ for $y\geq1$. In the next lemma, we identify $\Xi_{+}$ as the value of a scalar minimisation problem.

\begin{lemma}[Scalar minimisation]\label{lem:scalar-minimisation}
For every $y>0$,
\[
\inf_{0<x<1}\left[g\!\left(\frac{x}{y}\right)
+\frac{(x-y)^{2}}{y(1-x)}\right]=\Xi_{+}(y).
\]
\end{lemma}

\begin{proof}
The derivative of the expression in brackets with respect to $x$ is \(\frac{(x-y)(1-xy)}{xy(1-x)^{2}}\). Observe that the expression tends to $+\infty$ as $x\downarrow0$, and also as $x\uparrow1$ when $y\neq1$. If $y<1$, the unique critical point is $x=y$, and the value there is zero. If $y=1$, the expression is $-\log x$, whose infimum is zero. Finally, if $y>1$, the unique critical point in $(0,1)$ is $x=1/y$, and by direct calculation, we obtain the value $y+2\log y-1/y$.
\end{proof}

We also recall the following elementary fact on block matrices.

\begin{lemma}[Schur complement]\label{lem:schur}
Let $H\in\R^{n\times n}$ and $P\in\R^{k\times k}$ be symmetric, let $B\in\R^{n\times k}$, and assume that $P\succeq0$. Let $P^{\dagger}$ be the Moore--Penrose pseudoinverse of $P$. If
\[
\begin{pmatrix}
H&B\\
B^{\top}&-P
\end{pmatrix}\preceq0,
\]
then $Bz=0$ for every $z\in\ker P$, and
\[
H+BP^{\dagger}B^{\top}\preceq0.
\]
\end{lemma}

\begin{proof}
We take $z\in\ker P$ and evaluate the quadratic form at $(x,tz)$, which yields \(\langle Hx,x\rangle+2t\langle B^{\top}x,z\rangle\leq0 \qquad (t\in\R)\). Since $t$ is arbitrary, it follows that $B^{\top}x\perp\ker P$ for every $x$. In other terms, $Bz=0$ and $\ran B^{\top}\subset\ran P$, where $\ran$ denotes the range. We now evaluate at $(x,P^{\dagger}B^{\top}x)$ and use $PP^{\dagger}B^{\top}x=B^{\top}x$. This gives \(\langle Hx,x\rangle+\langle P^{\dagger}B^{\top}x,B^{\top}x\rangle\leq0\), which is the inequality (see also \cite{Albert1969}).
\end{proof}

In the next lemma, we present the dimension-free part of the proof. The idea is to bound the sum of the Bregman divergence of $-\log\det$ and a Schur complement term from below by a function of a single Rayleigh quotient.

\begin{lemma}[Coercivity]\label{lem:coercivity}
Let $A,C\in\R^{d\times d}$ be symmetric positive definite, let $v\in\Sd$, and set
\[
E:=C-A,
\qquad R:=A^{-1/2}CA^{-1/2},
\qquad x:=\langle Cv,v\rangle,
\qquad y:=\langle Av,v\rangle.
\]
We assume that $x\leq1$, and let $S\geq0$ satisfy
\begin{equation}\label{eq:coercivity-hyp}
\langle Ev,v\rangle^{2}\leq Sy(1-x).
\end{equation}
It follows that
\[
\mathcal{H}(R)+S\geq\Xi_{+}(y).
\]
Moreover, \eqref{eq:coercivity-hyp} holds in each of the following cases.
\begin{enumerate}[label=\textup{(\alph*)},leftmargin=2.8em]
\item $x<1$ and
\[
S=\frac{\langle A^{-1}Ev,Ev\rangle}{1-x}.
\]
\item $P\succeq0$ is symmetric, $\langle Cv,v\rangle+\langle Pv,v\rangle\leq1$, $Ez=0$ for every $z\in\ker P$, and
\[
S=\tr(A^{-1}EP^{\dagger}E).
\]
\end{enumerate}
\end{lemma}

\begin{proof}
Our goal is to replace the full spectrum of $R$ by its average in the direction $v$. To this end, let $(z_{i})$ be an orthonormal eigenbasis of $R$ with eigenvalues $(\lambda_{i})$, and set \(\theta_{i}:=\frac{\langle A^{1/2}v,z_{i}\rangle^{2}}{y}\). We have $\theta_{i}\geq0$, $\sum_{i}\theta_{i}=1$, and $\sum_{i}\theta_{i}\lambda_{i}=x/y$, since \(\langle RA^{1/2}v,A^{1/2}v\rangle=\langle Cv,v\rangle\). Given that $g$ is nonnegative and convex, it follows that
\[
\mathcal{H}(R)=\sum_{i}g(\lambda_{i})
\geq\sum_{i}\theta_{i}g(\lambda_{i})
\geq g(x/y).
\]
This is the spectral Jensen inequality from Section \ref{sec:proof-strategy}. Note that exactly at this step, the dimension ``disappears'' from the estimate. If $x=1$, then \eqref{eq:coercivity-hyp} gives $\langle Ev,v\rangle=0$, so $y=x=1$, and $\Xi_{+}(y)=0\leq\mathcal{H}(R)+S$. If $x<1$, the identity $\langle Ev,v\rangle=x-y$ and \eqref{eq:coercivity-hyp} yield \(S\geq\frac{(x-y)^{2}}{y(1-x)}\), and Lemma \ref{lem:scalar-minimisation} gives $\mathcal{H}(R)+S\geq\Xi_{+}(y)$.

It remains to verify \eqref{eq:coercivity-hyp} in the two cases. In case (a), we apply the Cauchy--Schwarz inequality:
\[
\langle Ev,v\rangle^{2}
=\langle A^{-1/2}Ev,A^{1/2}v\rangle^{2}
\leq\langle A^{-1}Ev,Ev\rangle y
=Sy(1-x).
\]
For case (b), note that $E$ is symmetric and vanishes on $\ker P$, so \(\ran E\subset(\ker P)^{\perp}=\ran P\). If $P^{\dagger/2}:=(P^{\dagger})^{1/2}$, then $Ev=P^{1/2}P^{\dagger/2}Ev$, and
\begin{align*}
|\langle Ev,v\rangle|
&=|\langle P^{\dagger/2}Ev,P^{1/2}v\rangle|\\
&\leq\|P^{\dagger/2}EA^{-1/2}\|_{\op}|A^{1/2}v|\,|P^{1/2}v|\\
&\leq\sqrt{S}\sqrt y\sqrt{1-x}.
\end{align*}
We used the fact that the operator norm is bounded by the Hilbert--Schmidt norm, along with \(\|P^{\dagger/2}EA^{-1/2}\|_{\mathrm{HS}}^{2} =\tr(A^{-1}EP^{\dagger}E)=S\) and $|P^{1/2}v|^{2}=\langle Pv,v\rangle\leq1-x$.
\end{proof}

\subsection{The profile estimate}

We now derive a translated-corrector estimate from a maximum principle in $(x,s)$. For $s\geq0$, let $\Qp(s)\geq1$ be the unique solution of
\begin{equation}\label{eq:Q}
\Xi(\Qp(s))=\frac{s^{2}}{2}.
\end{equation}
The function $\Qp$ is continuous and increasing, $\Qp(0)=1$, and, since $\Xi'(1)=4$,
\[
\Qp(s)=1+\frac{s^{2}}{8}+O(s^{4})
\qquad (s\downarrow0).
\]
Furthermore, $\Qp(s)\sim s^{2}/2$ as $s\to\infty$. We call $\Theta\in C^{2}([0,\infty))$ a penalty if
\[
\Theta(0)=\Theta'(0)=0,
\qquad \Theta''>0,
\qquad \Theta'(s)\to\infty\quad\text{as }s\to\infty.
\]

\begin{proposition}[Profile estimate]\label{prop:profile}
Let us assume \textup{(S1)}--\textup{(S3)} and that $W$ is convex on $\Omega$. We fix $e\in\Sd$, set $w:=w_{K}(e)>0$, and let $\Theta$ be a penalty such that
\begin{equation}\label{eq:profile-condition}
\Theta'(s)>s\Theta''(s)\Qp(\sqrt{\Lambda}s)
\end{equation}
for every $s>0$ with $\Theta'(s)\leq w$. It follows that
\[
G_{se}(x)\leq\Theta(s)
\qquad (x\in\R^{d},\ s\geq0).
\]
In particular,
\[
\sup_{\R^{d}}\Phi_{ee}\leq\Theta''(0).
\]
\end{proposition}

\begin{proof}
We may assume \eqref{eq:normalisation}. Let us explain the structure of the argument. We argue by contradiction and look at a maximum of the difference between the corrector and the penalty. At such a maximum, the first-order conditions express $\Phi_{ee}$ by means of $\Theta'$, the second-order conditions give an upper bound for the Bregman divergence plus a Schur complement term, and Lemma \ref{lem:coercivity} turns these two observations into a scalar inequality which contradicts \eqref{eq:profile-condition}.

We set \(Z(x,s):=G_{se}(x)-\Theta(s) \qquad (s\geq0)\). As \(G_{se}(x)=\int_{0}^{s}(s-r)\Phi_{ee}(x+re)\dd r\), the last conclusion follows from the first if we divide by $s^{2}/2$ and let $s\downarrow0$.

Suppose, for contradiction, that $Z(x_{1},s_{1})=2\eta>0$ for some $(x_{1},s_{1})$. For $\eps>0$, we consider \(Z_{\eps}(x,s):=G_{se}(x)-\Theta(s)-\eps\Phi(x)\). By Lemma \ref{lem:translated-corrector}(i) and $\Theta(s)-ws\to\infty$, we have $Z_{\eps}\to-\infty$ as $s\to\infty$, uniformly in $x$. The coercivity of $\Phi$ also gives $Z_{\eps}\to-\infty$ as $|x|\to\infty$ with $s$ bounded. It follows that $Z_{\eps}$ attains a maximum at some $(x_{\eps},s_{\eps})$. Moreover, \(Z_{\eps}(x_{\eps},s_{\eps})\geq Z_{\eps}(x_{1},s_{1})>\eta\) for $\eps<\eta/\Phi(x_{1})$ if $\Phi(x_{1})>0$, and for each $\eps$ otherwise. Given that $Z_{\eps}(x,0)=-\eps\Phi(x)\leq0$, we have $s_{\eps}>0$, that is, the maximum is attained at a nonzero translation.

We write
\[
\begin{aligned}
\bar A&:=A(x_{\eps}), & C&:=\mathrm{D}^{2}\Phi(x_{\eps}+s_{\eps}e), & R&:=\bar A^{-1/2}C\bar A^{-1/2},\\
u_{\eps}&:=\Phi_{ee}(x_{\eps}), & \bar T&:=T(x_{\eps}), & \bar s&:=s_{\eps}.
\end{aligned}
\]
Lemma \ref{lem:translated-corrector}(ii) gives the derivatives of $Z_{\eps}$ in $s$, and its derivatives in $x$ follow directly from the definition of $G_{se}$.

\textit{First-order conditions.}
At the maximum point, the identities $\nabla_{x}Z_{\eps}=0$ and $\partial_{s}Z_{\eps}=0$ give
\begin{equation}\label{eq:first-order}
\begin{aligned}
T(x_{\eps}+\bar s e)-\bar T&=\bar s\bar Ae+\eps\bar T,\\
\Phi_{e}(x_{\eps}+\bar s e)-\Phi_{e}(x_{\eps})&=\Theta'(\bar s).
\end{aligned}
\end{equation}
If we take the $e$-component of the first identity and insert it into the second, we obtain \(\bar s u_{\eps}=\Theta'(\bar s)-\eps\tau_{\eps}, \qquad \tau_{\eps}:=\langle\bar T,e\rangle\). This means that, up to an error of order $\eps$, the directional curvature at $x_{\eps}$ is determined by the slope of the penalty. The quantity $\tau_{\eps}$ is bounded by $\sup_{K}|y|$. Since the oscillation of $\Phi_{e}$ is at most $w$, the second identity also gives $\Theta'(\bar s)\leq w$.

\textit{Second-order conditions.}
The Hessian of $Z_{\eps}$ in $(x,s)$ is nonpositive at the maximum point. Its $(s,s)$ entry gives \(\langle Ce,e\rangle\leq\Theta''(\bar s)\), and its $(x,s)$ column is $(C-\bar A)e$. We apply Lemma \ref{lem:schur} to obtain \(\tr(\bar A^{-1}\mathrm{D}^{2}_{xx}Z_{\eps})+S\leq0\), where \(S:=\frac{\langle\bar A^{-1}(C-\bar A)e,(C-\bar A)e\rangle} {\Theta''(\bar s)-\langle Ce,e\rangle}\) if $\langle Ce,e\rangle<\Theta''(\bar s)$. If equality holds, then $(C-\bar A)e=0$, and we take $S:=0$. The Schur complement term $S$ is precisely the additional piece of information which a maximisation in $x$ alone would not provide.

It remains to compute the trace term. We have \(\mathrm{D}^{2}_{xx}Z_{\eps}=C-\bar A-\bar s A_{e}-\eps\bar A, \qquad A_{e}:=\partial_{e}A(x_{\eps})=\mathrm{D}^{2}\Phi_{e}(x_{\eps})\), and by \eqref{eq:MA},
\begin{align*}
\tr(\bar A^{-1}(C-\bar A))
&=\mathcal{H}(R)+\log\det R\\
&=\mathcal{H}(R)-V(x_{\eps}+\bar s e)+V(x_{\eps})
+W(T(x_{\eps}+\bar s e))-W(\bar T),\\
\tr(\bar A^{-1}A_{e})
&=-V_{e}(x_{\eps})+\langle\nabla W(\bar T),\bar Ae\rangle.
\end{align*}
If we use \eqref{eq:first-order} in the form \(\bar s\bar Ae=T(x_{\eps}+\bar s e)-\bar T-\eps\bar T\), we obtain
\begin{align*}
\tr(\bar A^{-1}\mathrm{D}^{2}_{xx}Z_{\eps})
&=\mathcal{H}(R)-D_{V}(x_{\eps}+\bar s e,x_{\eps})
+D_{W}(T(x_{\eps}+\bar s e),\bar T)\\
&\quad+\eps\langle\nabla W(\bar T),\bar T\rangle-\eps d.
\end{align*}
Here, $D_{V}\leq\Lambda\bar s^{2}/2$ and $D_{W}\geq0$, and convexity gives \(\langle\nabla W(\bar T),\bar T\rangle\geq W(\bar T)-W(0)\). With $C_{W}$ from Lemma \ref{lem:penalty-Phi}, we arrive at
\begin{equation}\label{eq:H-upper}
\mathcal{H}(R)+S\leq\frac{\Lambda\bar s^{2}}{2}+\eps C_{W}.
\end{equation}

\textit{The scalar inequality.}
We now apply Lemma \ref{lem:coercivity}(a) to $\bar A/\Theta''(\bar s)$ and $C/\Theta''(\bar s)$ with $v=e$. The bound on the $(s,s)$ entry guarantees that the hypothesis of the lemma is satisfied. When $\langle Ce,e\rangle=\Theta''(\bar s)$, the main statement of the lemma applies with $x=1$ and $S=0$ in its notation. If we set \(y_{\eps}:=\frac{u_{\eps}}{\Theta''(\bar s)}\), then \eqref{eq:H-upper} gives \(\Xi_{+}(y_{\eps})\leq\frac{\Lambda\bar s^{2}}{2}+\eps C_{W}\). Given the first-order conditions, we may rewrite this as
\begin{equation}\label{eq:scalar-profile}
\Xi_{+}\!\left(
\frac{\Theta'(s_{\eps})-\eps\tau_{\eps}}
{s_{\eps}\Theta''(s_{\eps})}
\right)
\leq\frac{\Lambda s_{\eps}^{2}}{2}+\eps C_{W}.
\end{equation}

\textit{Conclusion.}
The numbers $s_{\eps}$ belong to the compact set $\{s:\Theta'(s)\leq w\}$. Let $s_{*}$ be a limit point as $\eps\downarrow0$. If $s_{*}>0$, it follows from \eqref{eq:scalar-profile} that
\[
\Xi_{+}\!\left(\frac{\Theta'(s_{*})}{s_{*}\Theta''(s_{*})}\right)
\leq\frac{\Lambda s_{*}^{2}}{2}.
\]
This implies that \(\Theta'(s_{*})\leq s_{*}\Theta''(s_{*})\Qp(\sqrt{\Lambda}s_{*})\). It is immediate when $\Theta'(s_{*})/(s_{*}\Theta''(s_{*}))\leq1$, since $\Qp\geq1$, and in the opposite case, we use \eqref{eq:Q} and the monotonicity of $\Xi$. Since $\Theta'(s_{*})\leq w$, this inequality contradicts \eqref{eq:profile-condition}.

Now suppose that $s_{*}=0$. This is the stage where we apply the a priori bound \textup{(S3)}. For each $(x,t)$, along the corresponding sequence, we have
\[
Z(x,t)-\eps\Phi(x)
\leq Z_{\eps}(x_{\eps},s_{\eps})
\leq G_{s_{\eps}e}(x_{\eps})
\leq\frac{M_{\Phi}s_{\eps}^{2}}{2}\longrightarrow0.
\]
This gives $Z\leq0$, which contradicts $Z(x_{1},s_{1})>0$.
\end{proof}

The limiting profile follows from equality in \eqref{eq:profile-condition}. Let us set
\begin{equation}\label{eq:S}
\Sp(s):=s\exp\!\left(-\int_{0}^{s}\left(1-\frac1{\Qp(r)}\right)\frac{\dd r}{r}\right),
\qquad \Sp_{\infty}:=\lim_{s\to\infty}\Sp(s).
\end{equation}
We then have
\begin{equation}\label{eq:S-infty}
\Sp_{\infty}
=\exp\!\left(
\int_{0}^{1}\left(\frac1{\Qp(r)}-1\right)\frac{\dd r}{r}
+\int_{1}^{\infty}\frac{\dd r}{r\Qp(r)}
\right),
\end{equation}
and define
\begin{equation}\label{eq:C1}
C_{1}:=\frac1{\Sp_{\infty}}.
\end{equation}
The integrals converge because $1-1/\Qp(r)\sim r^{2}/8$ at zero and $1/\Qp(r)\sim2/r^{2}$ at infinity, and $\Sp$ is increasing because $(\log \Sp)'=1/(s\Qp)>0$. Numerical integration gives $\Sp_{\infty}=2.30543\ldots$ and
\begin{equation}\label{eq:C1-numeric}
C_{1}=0.43375\ldots.
\end{equation}
Recall that the conjectured sharp directional coefficient is $1/\sqrt{2\pi}$. We prove the estimate $C_{1}<0.463$ in Lemma \ref{lem:elementary-C1}.

\begin{proposition}[The constant $C_{1}$]\label{prop:C1}
Assume \textup{(S1)}--\textup{(S3)}, and that $W$ is convex on $\Omega$. It follows that, for every $e\in\Sd$,
\begin{equation}\label{eq:C1-bound}
\sup_{\R^{d}}\Phi_{ee}\leq C_{1}\sqrt{\Lambda}\,w_{K}(e).
\end{equation}
For a fixed direction $e$, the proof uses $V_{ee}\leq\Lambda_{e}$, so $\Lambda$ may be replaced by any $\Lambda_{e}>0$ that satisfies this directional estimate.
\end{proposition}

\begin{proof}
We set $w:=w_{K}(e)>0$ and fix $a>1/\Sp_{\infty}$. By Proposition \ref{prop:profile}, it suffices to construct a penalty that satisfies \eqref{eq:profile-condition} and \(\Theta''(0)=a\sqrt{\Lambda}\,w\). Our penalty follows the limiting profile $a\Sp$, which satisfies \eqref{eq:profile-condition} with equality, but it is bent slightly downwards by a Gaussian factor, so that the inequality becomes strict. Once the profile reaches the value $1$, the condition is no longer required, so we simply continue the profile linearly.

Let us choose $s_{1}>0$ with $a\Sp(s_{1})>1$, and then $\delta\in(0,1)$ such that \(\delta s_{1}^{2}\Qp(s_{1})<1,\allowbreak \qquad a\Sp(s_{1})\exp(-\delta s_{1}^{2}/2)>1\). Let $s_{2}\in(0,s_{1})$ be the point at which \(a\Sp(s_{2})\exp(-\delta s_{2}^{2}/2)=1\). We define
\[
p(\sigma):=
\begin{cases}
a\Sp(\sigma)\exp(-\delta\sigma^{2}/2),&0\leq\sigma\leq s_{2},\\
1+p'(s_{2})(\sigma-s_{2}),&\sigma>s_{2},
\end{cases}
\]
and
\[
\Theta(s):=\frac{w}{\sqrt{\Lambda}}
\int_{0}^{\sqrt{\Lambda}s}p(\sigma)\dd\sigma.
\]
On $(0,s_{2}]$, we have \(\frac{p'}p=\frac1{\sigma \Qp(\sigma)}-\delta\sigma\). Notice that this is positive because $\delta\sigma^{2}\Qp(\sigma)<1$, and it is strictly smaller than $1/(\sigma \Qp(\sigma))$. Moreover, $p(\sigma)/\sigma\to a$ as $\sigma\downarrow0$. This shows that $\Theta$ is a penalty and \(\Theta''(0)=a\sqrt{\Lambda}\,w\). We also have
\[
\Theta'(s)=wp(\sqrt{\Lambda}s),
\qquad \Theta''(s)=w\sqrt{\Lambda}\,p'(\sqrt{\Lambda}s).
\]
Observe that the inequality $\Theta'(s)\leq w$ holds exactly when $\sigma:=\sqrt{\Lambda}s\leq s_{2}$, and on this interval,
\[
\frac{\Theta'(s)}{s\Theta''(s)}
=\frac{p(\sigma)}{\sigma p'(\sigma)}
>\Qp(\sigma)=\Qp(\sqrt{\Lambda}s).
\]
This proves \eqref{eq:profile-condition}, and Proposition \ref{prop:profile} gives \(\sup_{\R^{d}}\Phi_{ee}\leq a\sqrt{\Lambda}\,w\). It now suffices to let $a\downarrow1/\Sp_{\infty}$.
\end{proof}

\begin{lemma}[An elementary bound]\label{lem:elementary-C1}
We have $\Sp_{\infty}\geq2^{11/6}e^{-1/2}$, that is, $C_{1}\leq2^{-11/6}\sqrt e=0.4626\ldots$.
\end{lemma}

\begin{proof}
Since $\log q\geq1-1/q$ holds, we have $\Xi(q)\geq q+2-3/q$ for $q\geq1$, so $\Qp(r)\leq q_{U}(r^{2}/2)$, where $q_{U}(z)$ is the positive root of $q+2-3/q=z$. We put this bound into \eqref{eq:S}, write $z=r^{2}/2$, and substitute \(t:=2q_{U}(z)=(z-2)+\sqrt{(z-2)^{2}+12}\), so that $2z=(t+6)(t-2)/t$ and $z=0$ corresponds to $t=2$. It follows that
\[
\log\Sp(s)
\geq\frac12\log\frac{(t+6)(t-2)}{t}
-\frac12\int_{2}^{t}\frac{\tau^{2}+12}{\tau^{2}(\tau+6)}\dd\tau
=\frac12\log\frac{(t+6)(t-2)}{t}
-\frac12\bigl(F(t)-F(2)\bigr),
\]
where $t=2q_{U}(s^{2}/2)$ and \(F(t):=-\frac13\log t-\frac2t+\frac43\log(t+6)\) is a primitive of the integrand. As $s\to\infty$, we have $t\to\infty$ and $\log\frac{(t+6)(t-2)}{t}-F(t)\to0$. Therefore, $\log\Sp_{\infty}\geq\frac12F(2)=\frac{11}{6}\log2-\frac12$.
\end{proof}

\subsection{The semi-log-concave estimate}

For targets with $\mathrm{D}^{2}W\succeq-\rho\Id$, we use a vector translation $m\in\R^{d}$. The reason is that negative target curvature creates a term in Lemma \ref{lem:translated-corrector} which involves the full increment of $T$, and stationarity in a vector translation controls this increment in all directions. For $r\geq0$ and $a>0$, let $p_{a,r}$ solve
\begin{equation}\label{eq:par}
p_{a,r}'(s)=\frac{p_{a,r}(s)}{s\Qp_{r}(s,p_{a,r}(s))},
\qquad
\Qp_{r}(s,z):=\Xi^{-1}\!\left(\frac{s^{2}}2+\frac{rz^{2}}2\right),
\qquad
\lim_{s\downarrow0}\frac{p_{a,r}(s)}s=a,
\end{equation}
where $\Xi^{-1}$ is the inverse of $\Xi$ on $[1,\infty)$. In particular, $\Qp_{0}(s,z)=\Qp(s)$ and $p_{a,0}=a\Sp$. Existence and uniqueness follow from the fact that $q:=p_{a,r}/s$ solves the regular equation
\[
q'=-\frac qs\left(1-\frac1{\Qp_{r}(s,sq)}\right),
\qquad q(0)=a.
\]
The function $p_{a,r}$ is increasing and $p_{a,r}\leq as$. By comparison, $\sup_{s}p_{a,r}(s)$ is also nondecreasing in $a$ and nonincreasing in $r$. We define
\begin{equation}\label{eq:C2}
C_{2}(r):=\inf\left\{a>0:\sup_{s>0}p_{a,r}(s)>1\right\}
\in[C_{1},\infty].
\end{equation}
Note that $C_{2}(r)$ is precisely the threshold value of the initial slope above which $p_{a,r}$ exceeds $1$.

\begin{proposition}[Semi-log-concave estimate]\label{prop:semilog-estimate}
We assume \textup{(S1)}--\textup{(S3)} and $\mathrm{D}^{2}W\succeq-\rho\Id$ on $\Omega$ for some $\rho\geq0$. If $r:=\rho\diam(K)^{2}$, then
\begin{equation}\label{eq:Mphi-C2}
M_{\Phi}\leq C_{2}(r)\sqrt{\Lambda}\,\diam(K).
\end{equation}
\end{proposition}

\begin{proof}
The proof follows the same idea as that of Proposition \ref{prop:profile}. We first normalise the data. If we replace $\Phi$ by \(\sqrt{\Lambda}\,\diam(K)^{-1}\Phi(\cdot/\sqrt{\Lambda})\), the new gradient transports the image of $\mu$ under $x\mapsto\sqrt{\Lambda}x$ to the image of $\nu$ under $y\mapsto y/\diam(K)$, and $\mathrm{D}^{2}\Phi$ is multiplied by $\Lambda^{-1/2}\diam(K)^{-1}$, so we may assume \(\Lambda=1, \qquad \diam(K)=1, \qquad \rho=r\), with \eqref{eq:normalisation}. Notice that $|\nabla\Phi|\leq1$ on $\R^{d}$, because $\nabla\Phi$ takes values in $K$ and $0\in\Omega$.

Let us fix $a>0$ with $\sup p_{a,r}>1$. It suffices to prove that $M_{\Phi}\leq a$. As in the proof of Proposition \ref{prop:C1}, we bend the profile slightly. For $\delta>0$, we set \(p(s):=p_{a,r}(s)e^{-\delta s^{2}/2} \qquad (0\leq s\leq s_{2})\), where $s_{2}$ is the first point at which $p_{a,r}(s_{2})e^{-\delta s_{2}^{2}/2}=1$. Such a point exists for small $\delta$ because $\sup p_{a,r}>1$. Since \((\log p)'=\frac1{s\Qp_{r}(s,p_{a,r}(s))}-\delta s\) and $s_{2}$ remains bounded, we may also choose $\delta$ so small that $p'>0$ on $[0,s_{2}]$. We extend $p$ linearly beyond $s_{2}$ with slope $p'(s_{2})>0$ and set \(\Theta(s):=\int_{0}^{s}p(t)\dd t\). It follows that $\Theta$ is a penalty, $\Theta''(0)=a$, and for $0<s\leq s_{2}$, the inequalities $p\leq p_{a,r}$ and monotonicity of $\Qp_{r}$ in its second variable give
\begin{equation}\label{eq:semilog-profile}
\frac{\Theta'(s)}{s\Theta''(s)}
=\frac{p}{sp'}
=\left(\frac1{\Qp_{r}(s,p_{a,r}(s))}-\delta s^{2}\right)^{-1}
>\Qp_{r}(s,p_{a,r}(s))
\geq \Qp_{r}(s,\Theta'(s)).
\end{equation}

For $m\in\R^{d}$, let us consider \(Z_{\eps}(x,m):=G_{m}(x)-\Theta(|m|)-\eps\Phi(x)\). Recall that $w_{K}(m)\leq|m|$, so, by Lemma \ref{lem:translated-corrector}(i), this function attains its maximum at some $(x_{\eps},m_{\eps})$. Suppose that \(\sup_{x,m}\bigl(G_{m}(x)-\Theta(|m|)\bigr)>0\). In this case, $m_{\eps}\neq0$ for small $\eps$. We write
\[
\begin{aligned}
s&:=|m_{\eps}|, & \widehat m&:=m_{\eps}/s, & \bar A&:=A(x_{\eps}),\\
C&:=A(x_{\eps}+m_{\eps}), & E&:=C-\bar A, & R&:=\bar A^{-1/2}C\bar A^{-1/2},\\
&&\bar T&:=T(x_{\eps}).
\end{aligned}
\]
Stationarity in $x$ and $m$ yields \(T(x_{\eps}+m_{\eps})-\bar T =\bar Am_{\eps}+\eps\bar T =\Theta'(s)\widehat m\). In particular, \(|T(x_{\eps}+m_{\eps})-\bar T|=\Theta'(s)\leq1\) and
\[
\langle\bar A\widehat m,\widehat m\rangle
=\frac{\Theta'(s)}s-\frac{\eps\langle\bar T,\widehat m\rangle}{s}.
\]
The Hessian in $(x,m)$ is nonpositive at the maximum point. Its $(m,m)$ block is \(C-\mathrm{D}^{2}_{mm}[\Theta(|m|)]\), where
\[
\mathrm{D}^{2}_{mm}[\Theta(|m|)]
=\Theta''(s)\widehat m\otimes\widehat m
+\frac{\Theta'(s)}s(\Id-\widehat m\otimes\widehat m)
=:\Pi,
\]
and its $(x,m)$ block is $E$. Let us now set $P:=\Pi-C\succeq0$. By Lemma \ref{lem:schur}, we have $Ez=0$ for $z\in\ker P$ and \(\tr(\bar A^{-1}\mathrm{D}^{2}_{xx}Z_{\eps})+S\leq0, \qquad S:=\tr(\bar A^{-1}EP^{\dagger}E)\). As in the proof of Proposition \ref{prop:profile}, with $A_{m_{\eps}}:=\partial_{m_{\eps}}A(x_{\eps})$, it holds that
\begin{align*}
\tr(\bar A^{-1}\mathrm{D}^{2}_{xx}Z_{\eps})
&=\tr(\bar A^{-1}(C-\bar A-A_{m_{\eps}}-\eps\bar A))\\
&=\mathcal{H}(R)-D_{V}(x_{\eps}+m_{\eps},x_{\eps})
+D_{W}(T(x_{\eps}+m_{\eps}),\bar T)\\
&\quad+\eps\langle\nabla W(\bar T),\bar T\rangle-\eps d.
\end{align*}
Moreover,
\[
D_{V}\leq\frac{s^{2}}2,
\qquad
D_{W}\geq-\frac r2|T(x_{\eps}+m_{\eps})-\bar T|^{2}
=-\frac r2\Theta'(s)^{2}.
\]
The convexity of $W+\frac r2|y|^{2}$ and $|\bar T|\leq1$ also give \(\langle\nabla W(\bar T),\bar T\rangle \geq W(\bar T)-W(0)-\frac r2\). These bounds yield
\[
\mathcal{H}(R)+S
\leq\frac{s^{2}}2+\frac r2\Theta'(s)^{2}+\eps C_{\mathrm{pen}},
\qquad
C_{\mathrm{pen}}:=d+W(0)-\inf_{\Omega}W+\frac r2.
\]

We now apply Lemma \ref{lem:coercivity}(b) to $\bar A/\Theta''(s)$, $C/\Theta''(s)$, and $P/\Theta''(s)$ with $v=\widehat m$. Since \(\langle Cv,v\rangle+\langle Pv,v\rangle =\langle\Pi v,v\rangle=\Theta''(s)\), the hypotheses of case (b) are satisfied. We then obtain
\[
\mathcal{H}(R)+S\geq\Xi_{+}(y_{\eps}),
\qquad
y_{\eps}:=\frac{\langle\bar A\widehat m,\widehat m\rangle}{\Theta''(s)}.
\]
Let $\eps\downarrow0$ along a sequence with $|m_{\eps}|\to s_{*}$. If $s_{*}>0$, then
\[
\Xi_{+}\!\left(\frac{\Theta'(s_{*})}{s_{*}\Theta''(s_{*})}\right)
\leq\frac{s_{*}^{2}}2+\frac r2\Theta'(s_{*})^{2}.
\]
In other terms, \(\Theta'(s_{*}) \leq s_{*}\Theta''(s_{*})\Qp_{r}(s_{*},\Theta'(s_{*}))\). Notice that $\Theta'(s_{*})\leq1$ forces $s_{*}\leq s_{2}$, so this contradicts \eqref{eq:semilog-profile}. If $s_{*}=0$, the argument from the proof of Proposition \ref{prop:profile} gives $G_{m}\leq\Theta(|m|)$.

In both cases, we conclude that \(G_{m}(x)\leq\Theta(|m|) \qquad (x,m\in\R^{d})\). If we let $m=te$ with $e\in\Sd$ and $t\downarrow0$, we obtain \(\langle \mathrm{D}^{2}\Phi(x)e,e\rangle \leq\Theta''(0)=a\), which proves $M_{\Phi}\leq a$.
\end{proof}

\begin{proposition}[Properties of $C_{2}$]\label{prop:C2}
The function $C_{2}$ is nondecreasing on $[0,\infty)$, $C_{2}(0)=C_{1}$, and $C_{2}(r)<\infty$ for every $r$. Moreover,
\begin{equation}\label{eq:C2-growth}
C_{2}(r)\leq e^{3/4}r^{1/4}e^{r/4}
\qquad (r\geq1).
\end{equation}
\end{proposition}

\begin{proof}
Monotonicity in $r$ follows from the comparison after \eqref{eq:par}, and $C_{2}(0)=1/\Sp_{\infty}$ because $p_{a,0}=a\Sp$ and $\Sp$ increases to $\Sp_{\infty}$.

For the growth bound, let $r\geq1$ and take \(a>e^{3/4}r^{1/4}e^{r/4}\), so that $a\sqrt r\geq1$. We prove that $\sup p_{a,r}>1$. To this end, we compare $p_{a,r}$ with the solution of a simpler equation, which can be integrated by hand. Given that $\Xi^{-1}(z)\leq1+z$, which follows from $\Xi(1+z)\geq z$, the solution $\bar p$ of
\[
\bar p'=\frac{\bar p}{s(1+s^{2}/2+r\bar p^{2}/2)},
\qquad \frac{\bar p(s)}s\longrightarrow a,
\]
satisfies $\bar p\leq p_{a,r}$ by comparison. It is then enough to show that $\sup\bar p>1$.

Let us write $\bar p=sq$. Since $q\leq a$, we have \(q'=-\frac{qs(1+rq^{2})}{2+s^{2}+rs^{2}q^{2}} \geq-\frac12(1+ra^{2})sq\), and so \(\bar p(s)\geq as e^{-(1+ra^{2})s^{2}/4} \qquad (s\geq0)\). At \(s_{0}:=\frac1{a\sqrt r}\leq1\), we get \(\bar p(s_{0})\geq r^{-1/2}e^{-1/4-1/(4ra^{2})} \geq r^{-1/2}e^{-1/2}=:p_{0}\). Suppose that $\bar p\leq1$ on $[s_{0},1]$. On this interval, \(\bar p'\geq\frac{\bar p}{s(3/2+r\bar p^{2}/2)}\), which can be rewritten as \(\left(\frac32\log\bar p+\frac r4\bar p^{2}\right)' \geq\frac1s\). Integration from $s_{0}$ to $1$ gives
\begin{align*}
\log(a\sqrt r)=\log\frac1{s_{0}}
&\leq\frac32\log\frac{\bar p(1)}{\bar p(s_{0})}
+\frac r4\bigl(\bar p(1)^{2}-\bar p(s_{0})^{2}\bigr)\\
&\leq\frac32\log\frac1{p_{0}}+\frac r4
=\frac34+\frac34\log r+\frac r4.
\end{align*}
It follows that \(\log a\leq\frac r4+\frac14\log r+\frac34\), contrary to the choice of $a$. Therefore, the increasing function $\bar p$ exceeds $1$, and so does $p_{a,r}\geq\bar p$. Since $C_{2}$ is nondecreasing, this also shows that $C_{2}(r)<\infty$ for every $r$.
\end{proof}

\section{Approximation and proofs of the main theorems}\label{sec:approximation}

\subsection{Brenier potentials and their stability}\label{sec:stability}

We first analyse the equivalence between distributional and second-difference bounds, which we used in the introduction. Let $\Phi$ be finite and convex on $\R^{d}$, let $v\in\R^{d}$, and let us fix $c\in\R$. Suppose that $\partial_{vv}\Phi\leq c$ in the sense of distributions. In this case, the mollifications
\[
\Phi_{\eps}:=\Phi*\eta_{\eps}
\]
satisfy $\partial_{vv}\Phi_{\eps}\leq c$ pointwise. Here, $\eta\in C_{c}^{\infty}(\R^{d})$ is nonnegative and radial, $\int\eta=1$, $\supp\eta\subset B_{1}(0)$, and $\eta_{\eps}:=\eps^{-d}\eta(\cdot/\eps)$. This means that
\[
t\longmapsto \Phi_{\eps}(x+tv)-\frac{c}{2}t^{2}
\]
is concave, so we obtain \eqref{eq:second-difference} for $\Phi_{\eps}$, and then for $\Phi$ after we let $\eps\downarrow0$.

Conversely, let us suppose that \eqref{eq:second-difference} holds. For every $0\leq\varphi\in C_{c}^{\infty}(\R^{d})$,
\begin{align*}
\int_{\R^{d}}\Phi\,\partial_{vv}\varphi\dd x
&=\lim_{h\to0}\int_{\R^{d}}\Phi(x)
\frac{\varphi(x+hv)+\varphi(x-hv)-2\varphi(x)}{h^{2}}\dd x\\
&=\lim_{h\to0}\int_{\R^{d}}
\frac{\Phi(x+hv)+\Phi(x-hv)-2\Phi(x)}{h^{2}}\varphi(x)\dd x\\
&\leq c\int_{\R^{d}}\varphi\dd x.
\end{align*}
This proves that the two conditions are equivalent, and in the rest of the paper, we use the second-difference version.

\begin{lemma}[Range and uniqueness]\label{lem:range-uniqueness}
Let $\mu\in\Ptwo(\R^{d})$ have a Lebesgue density which is positive almost everywhere. Let $\nu$ be a probability measure with $\nu(K)=1$ for a compact convex set $K$, and let $\Psi:\R^{d}\to\R$ be convex with $(\nabla\Psi)_{\#}\mu=\nu$. It follows that
\[
\partial\Psi(\R^{d})\subset K.
\]
Moreover, any two such functions $\Psi$ differ by a constant.
\end{lemma}

\begin{proof}
The positivity of the source density and the transport identity give $\nabla\Psi\in K$ almost everywhere. The idea is to show that $\Psi$ coincides with the Legendre transform of $\Psi^{*}$ restricted to $K$, whose subgradients lie in $K$ by construction. To this end, with $\iota_{K}$ the convex indicator function of $K$, we set
\[
F:=\Psi^{*}+\iota_{K},
\qquad
\Psi_{K}:=F^{*}=\sup_{y\in K}\bigl\{\langle\cdot,y\rangle-\Psi^{*}(y)\bigr\}.
\]
We have $\Psi_{K}\leq\Psi^{**}=\Psi$, while Fenchel equality gives the reverse inequality almost everywhere. Since both functions are finite and continuous, we conclude that $\Psi_{K}=\Psi$ on $\R^{d}$.

Now let $p\in\partial F^{*}(x)$. Fenchel equality gives \(F(p)+F^{*}(x)=\langle x,p\rangle\), so $p\in\operatorname{dom}F\subset K$, which proves that $\partial\Psi(\R^{d})\subset K$. For uniqueness, recall that, by Brenier--McCann uniqueness \cite{Brenier1991,McCann1995}, two such gradients agree $\mu$-almost everywhere, and hence Lebesgue-almost everywhere. This implies that their potentials differ by a constant.
\end{proof}

\begin{lemma}[Stability]\label{lem:stability}
Let $\mu_{n},\mu\in\Ptwo(\R^{d})$ have Lebesgue densities which are positive almost everywhere, and assume that $\mu_{n}\to\mu$ in $W_{2}$. Let $\nu_{n},\nu$ be probability measures supported in a fixed compact convex set $C$, with $\nu_{n}\rightharpoonup\nu$ weakly. Let $\Phi_{n}$ and $\Phi$ be convex functions such that
\[
(\nabla\Phi_{n})_{\#}\mu_{n}=\nu_{n},
\qquad
(\nabla\Phi)_{\#}\mu=\nu,
\qquad
\Phi_{n}(0)=\Phi(0)=0.
\]
It follows that $\Phi_{n}\to\Phi$ locally uniformly on $\R^{d}$. In particular, let $v\in\R^{d}$, let $c_{n},c\in\R$ with $\limsup_{n\to\infty}c_{n}\leq c$, and suppose that
\[
\Phi_{n}(x+hv)+\Phi_{n}(x-hv)-2\Phi_{n}(x)\leq c_{n}h^{2}
\]
for all $x,h$ and $n$. We then have
\[
\Phi(x+hv)+\Phi(x-hv)-2\Phi(x)\leq ch^{2}
\]
for all $x,h$.
\end{lemma}

\begin{proof}
By Lemma \ref{lem:range-uniqueness}, we know that $\partial\Phi_{n}(\R^{d})\subset C$, so the normalised potentials are equi-Lipschitz. In this case, every subsequence has a further subsequence which converges locally uniformly to a convex function $\overline\Phi$ that satisfies $\overline\Phi(0)=0$. We must still identify this limit as $\Phi$.

Along this subsequence, the plans \(\pi_{n}:=(\operatorname{id}_{\R^{d}},\nabla\Phi_{n})_{\#}\mu_{n}\) are tight because their marginals converge weakly. Therefore, we may assume that $\pi_{n}\rightharpoonup\pi$, where $\pi$ has marginals $\mu$ and $\nu$. Let us show that $\pi$ is concentrated on the graph of $\partial\overline\Phi$. Notice that $\supp\pi_{n}\subset\operatorname{graph}(\partial\Phi_{n})$. If $(x,y)\in\supp\pi$, the Portmanteau theorem gives indices $n_{k}\to\infty$ and points $(x_{k},y_{k})\in\supp\pi_{n_{k}}$ such that $(x_{k},y_{k})\to(x,y)$. We pass to the limit in \(\Phi_{n_{k}}(z)\geq\Phi_{n_{k}}(x_{k})+\langle y_{k},z-x_{k}\rangle\) and obtain $y\in\partial\overline\Phi(x)$.

Given that the graph of $\partial\overline\Phi$ is cyclically monotone, $\pi$ is an optimal plan by \cite[Theorem 5.10]{Villani2009}, and Brenier--McCann uniqueness gives \(\pi=(\operatorname{id}_{\R^{d}},\nabla\Phi)_{\#}\mu\). This implies that $\nabla\Phi\in\partial\overline\Phi$ $\mu$-almost everywhere. We conclude that $\nabla\Phi=\nabla\overline\Phi$ almost everywhere, and the normalisation yields $\Phi=\overline\Phi$. Given that the limit is independent of the subsequence, the whole sequence converges locally uniformly. Finally, the second-difference inequality passes to the limit pointwise.
\end{proof}

\begin{lemma}[Lipschitz gradients]\label{lem:lipschitz-gradients}
Let $\Psi:\R^{d}\to\R$ be convex, and let $L\geq0$ satisfy
\[
\Psi(x+h)+\Psi(x-h)-2\Psi(x)\leq L|h|^{2}
\qquad (x,h\in\R^{d}).
\]
It follows that $\Psi\in C^{1,1}(\R^{d})$ and
\[
|\nabla\Psi(x)-\nabla\Psi(y)|\leq L|x-y|
\qquad (x,y\in\R^{d}).
\]
\end{lemma}

\begin{proof}
The mollifications $\Psi_{\eps}:=\Psi*\eta_{\eps}$ are smooth and convex, and they satisfy the same second-difference inequality. We obtain \(0\preceq \mathrm{D}^{2}\Psi_{\eps}\preceq L\Id\), so $\nabla\Psi_{\eps}$ is $L$-Lipschitz. We also have $\Psi_{\eps}\to\Psi$ locally uniformly as $\eps\downarrow0$, and by convexity, the gradients $\nabla\Psi_{\eps}$ are locally uniformly bounded. By the Arzel\`a--Ascoli theorem, along a subsequence, $\nabla\Psi_{\eps}\to G$ locally uniformly, where $G$ is $L$-Lipschitz. If we pass to the limit in the fundamental theorem of calculus, we obtain \(\Psi(y)-\Psi(x)=\int_{0}^{1}\langle G(x+t(y-x)),y-x\rangle\dd t\) for all $x,y$. It follows that $\Psi$ is differentiable and $\nabla\Psi=G$.
\end{proof}

\subsection{Log-concave targets}

\begin{proposition}[Removing the Hessian bound]\label{prop:remove-hessian}
Assume \textup{(S1)} and \textup{(S2)}, and suppose that $W$ is convex on $\Omega$. Let $\Phi$ be convex with $(\nabla\Phi)_{\#}\mu=\nu$. It follows that, for every $e\in\Sd$,
\[
\partial_{ee}\Phi\leq C_{1}\sqrt\Lambda\,w_{K}(e)
\]
in the sense of distributions.
\end{proposition}

\begin{proof}
The intuition is as follows. We first add a small amount of uniform convexity to the target, which gives the qualitative Hessian bound \textup{(S3)} by Caffarelli's theorem. We then apply Proposition \ref{prop:C1}, where the bound does not depend on the added convexity, and finally, we remove the perturbation by stability.

We may assume that $0\in\Omega$. For $\delta>0$, let
\[
\nu_{\delta}:=Z_{\delta}^{-1}
\exp\left(-W-\frac{\delta}{2}|y|^{2}\right)\1_{\Omega}\dd y,
\]
and let $\Phi_{\delta}$ be the convex function such that $(\nabla\Phi_{\delta})_{\#}\mu=\nu_{\delta}$ and $\Phi_{\delta}(0)=0$. The potential \(W_{\delta}:=W+\frac{\delta}{2}|y|^{2}\) is smooth on $U$ and convex on $\Omega$, so \textup{(S2)} holds for $\nu_{\delta}$, and Proposition \ref{prop:regularity} applies.

We now verify \textup{(S3)} by Caffarelli's contraction theorem, as given in \cite[Theorem 1]{FathiGozlanProdhomme2020}. To this end, let us define \(S_{\Lambda}(x):=\sqrt\Lambda x, \qquad S_{\delta}(y):=\sqrt\delta y\). The measure $(S_{\Lambda})_{\#}\mu$ has density proportional to $\exp(P_{\mu})$ with respect to $\gamma_{d}$, where \(P_{\mu}(x):=\frac12|x|^{2}-V(x/\sqrt\Lambda)\). This function is convex because $\mathrm{D}^{2}V\preceq\Lambda\Id$, and the measure has finite second moment. The measure $(S_{\delta})_{\#}\nu_{\delta}$ has density proportional to $\exp(-P_{\nu})$ with respect to $\gamma_{d}$, where
\[
P_{\nu}(y):=W(y/\sqrt\delta)
\quad\text{for }y\in\sqrt\delta\,\Omega,
\qquad
P_{\nu}:=+\infty
\quad\text{otherwise}.
\]
The function $P_{\nu}$ is convex, and the measure is compactly supported. Therefore, the contraction theorem gives a convex function $\varphi\in C^{1}(\R^{d})$ whose gradient is $1$-Lipschitz and satisfies \((\nabla\varphi)_{\#}(S_{\Lambda})_{\#}\mu=(S_{\delta})_{\#}\nu_{\delta}\). We set \(\widetilde\varphi(x):=(\Lambda\delta)^{-1/2}\varphi(\sqrt\Lambda x)\). It follows that $\widetilde\varphi$ is convex, \(\nabla\widetilde\varphi =S_{\delta}^{-1}\circ\nabla\varphi\circ S_{\Lambda}\), and $\nabla\widetilde\varphi$ pushes $\mu$ to $\nu_{\delta}$ and is $\sqrt{\Lambda/\delta}$-Lipschitz. By Lemma \ref{lem:range-uniqueness}, we have $\nabla\widetilde\varphi=\nabla\Phi_{\delta}$ almost everywhere, and then everywhere by continuity. We obtain \(M_{\Phi_{\delta}}\leq\sqrt{\Lambda/\delta}<\infty\), and Proposition \ref{prop:C1} now yields \(\partial_{ee}\Phi_{\delta}\leq C_{1}\sqrt\Lambda\,w_{K}(e)\) pointwise, and so in second-difference form. Note that this bound does not depend on $\delta$.

Given that $K$ is bounded, $\nu_{\delta}\to\nu$ in total variation as $\delta\downarrow0$. By Lemma \ref{lem:stability}, with $\mu_{n}=\mu$, we obtain local uniform convergence to the potential $\Phi$ normalised by $\Phi(0)=0$, and the second-difference estimate passes to the limit.
\end{proof}

\begin{lemma}[Inner approximation]\label{lem:inner-approximation}
Let $\nu$ be a log-concave probability measure whose support $K$ is a convex body, and set $\Omega:=\operatorname{int}K$. There exist bounded open convex sets $\Omega_{n}\Subset\Omega$, with $K_{n}:=\overline{\Omega_{n}}$, and convex functions $W_{n}$ which are smooth on a neighbourhood of $K_{n}$, such that
\[
\nu_{n}:=Z_{n}^{-1}e^{-W_{n}}\1_{\Omega_{n}}\dd y\longrightarrow\nu
\]
in total variation. Moreover,
\[
w_{K_{n}}(e)\leq w_{K}(e)
\qquad (e\in\Sd),
\qquad
\diam(K_{n})\leq\diam(K).
\]
\end{lemma}

\begin{proof}
By Borell's characterisation \cite{Borell1975}, we have \(\nu=Z_{\nu}^{-1}e^{-W}\1_{\Omega}\dd y\) for a finite convex function $W$ on $\Omega$. We shrink the support homothetically and mollify the potential. Let us fix $y_{0}\in\Omega$, choose $\alpha_{n}\uparrow1$ in $(0,1)$, and let \(\Omega_{n}:=y_{0}+\alpha_{n}(\Omega-y_{0})\). We then have \(K_{n}=y_{0}+\alpha_{n}(K-y_{0})\Subset\Omega\), the sets $\Omega_{n}$ increase to $\Omega$, and every width and the diameter are multiplied by $\alpha_{n}$.

As $W$ is continuous on $\Omega$, we may choose \(0<\eps_{n}<\frac12\operatorname{dist}(K_{n},\partial\Omega)\) such that $W_{n}:=W*\eta_{\eps_{n}}$ is smooth and convex on a neighbourhood of $K_{n}$ and \(\sup_{\Omega_{n}}|W_{n}-W|\leq\frac1n\). Let $f:=e^{-W}\1_{\Omega}$ and $f_{n}:=e^{-W_{n}}\1_{\Omega_{n}}$. On $\Omega_{n}$, we have \(e^{-1/n}f\leq f_{n}\leq e^{1/n}f\), and it follows that
\[
\|f_{n}-f\|_{L^{1}}
\leq (e^{1/n}-1)\int_{\Omega}f
+\int_{\Omega\setminus\Omega_{n}}f
\longrightarrow0.
\]
After normalisation, this gives convergence in total variation.
\end{proof}

\begin{proposition}[Smooth sources]\label{prop:smooth-sources}
Assume \textup{(S1)}, and let $\nu$ be a compactly supported log-concave probability measure with support $K$, of arbitrary affine dimension, and let $\Phi$ be convex with $(\nabla\Phi)_{\#}\mu=\nu$. It follows that, for every $e\in\Sd$,
\[
\partial_{ee}\Phi\leq C_{1}\sqrt\Lambda\,w_{K}(e)
\]
in the sense of distributions. Moreover,
\[
\Phi\in C^{1,1}(\R^{d}),
\qquad
\nabla\Phi(\R^{d})\subset K,
\qquad
\Lip(\nabla\Phi)\leq C_{1}\sqrt\Lambda\,\diam(K).
\]
\end{proposition}

\begin{proof}
We normalise $\Phi(0)=0$ and suppose first that $K$ is a convex body. Let $\nu_{n}$ and $K_{n}$ be given by Lemma \ref{lem:inner-approximation}, and let $\Phi_{n}$ be the convex function such that $(\nabla\Phi_{n})_{\#}\mu=\nu_{n}$ and $\Phi_{n}(0)=0$. We apply Proposition \ref{prop:remove-hessian} to obtain
\[
\partial_{ee}\Phi_{n}
\leq C_{1}\sqrt\Lambda\,w_{K_{n}}(e)
\leq C_{1}\sqrt\Lambda\,w_{K}(e).
\]
All supports are in $K$, and $\nu_{n}\rightharpoonup\nu$, so Lemma \ref{lem:stability} gives the directional estimate for $\Phi$.

Let us now consider an arbitrary affine dimension. We reduce it to the previous case by thickening a lower-dimensional target in the orthogonal directions, which costs only an arbitrarily small amount of width. After a translation, we may assume that the affine hull of $K$ is a linear subspace $E$. Let $\sigma_{\eps}$ be normalised Lebesgue measure on \(\{z\in E^{\perp}:|z|\leq\eps\}\), and let us define \(\nu_{\eps}:=[(y,z)\mapsto y+z]_{\#}(\nu\otimes\sigma_{\eps})\). This measure is log-concave because its density with respect to Lebesgue measure on $E\oplus E^{\perp}=\R^{d}$ is a product of log-concave densities \cite{Borell1975}. Its support \(K_{\eps}:=K+\{z\in E^{\perp}:|z|\leq\eps\}\) is a convex body contained in $K+\{z\in E^{\perp}:|z|\leq1\}$ for $\eps\leq1$, and additivity of support functions gives \(w_{K_{\eps}}(e)=w_{K}(e)+2\eps|P_{E^{\perp}}e|\), where $P_{E^{\perp}}$ is the orthogonal projection onto $E^{\perp}$. Moreover, $\nu_{\eps}\rightharpoonup\nu$ as $\eps\downarrow0$. By the first part, \(\partial_{ee}\Phi_{\eps}\leq C_{1}\sqrt\Lambda\,w_{K_{\eps}}(e)\) for the normalised potentials $\Phi_{\eps}$ from $\mu$ to $\nu_{\eps}$, and Lemma \ref{lem:stability} passes this estimate to $\Phi$.

Finally, since $w_{K}(e)\leq\diam(K)$, we have
\[
\Phi(x+h)+\Phi(x-h)-2\Phi(x)
\leq C_{1}\sqrt\Lambda\,\diam(K)|h|^{2}
\qquad (x,h\in\R^{d}).
\]
By Lemma \ref{lem:lipschitz-gradients}, we obtain the $C^{1,1}$ regularity and the stated Lipschitz bound. Since $\Phi$ is differentiable everywhere, Lemma \ref{lem:range-uniqueness} gives \(\nabla\Phi(\R^{d})=\partial\Phi(\R^{d})\subset K\). \end{proof}

\subsection{Nonsmooth sources}

For $t>0$, let $\gamma_{d,t}$ denote the centred Gaussian measure on $\R^{d}$ with covariance $t\Id$.

\begin{lemma}[Gaussian smoothing of the source]\label{lem:source-smoothing}
Let $Q\succ0$, let $V:\R^{d}\to\R$ be such that
\[
x\longmapsto\frac12\langle Qx,x\rangle-V(x)
\]
is convex, and suppose that $e^{-V}$ is integrable and $\mu:=Z_{\mu}^{-1}e^{-V}\dd x\in\Ptwo(\R^{d})$.
For $t>0$, the measure $\mu_{t}:=\mu*\gamma_{d,t}$ has a positive density $e^{-V_{t}}$ with $V_{t}\in C^{\infty}(\R^{d})$, and
\[
\mathrm{D}^{2}V_{t}\preceq Q(\Id+tQ)^{-1}\preceq Q.
\]
Moreover, $\mu_{t}$ has finite second moment and
\[
W_{2}(\mu_{t},\mu)\leq\sqrt{td}.
\]
\end{lemma}

\begin{proof}
By differentiation under the integral sign, the density
\[
p_{t}(x):=Z_{\mu}^{-1}(2\pi t)^{-d/2}
\int_{\R^{d}}e^{-V(y)}e^{-|x-y|^{2}/(2t)}\dd y
\]
is positive and smooth. We factor out a Gaussian with precision $Q(\Id+tQ)^{-1}$ and show that the remaining factor is log-convex. To this end, we set
\[
\begin{aligned}
P(y)&:=\frac12\langle Qy,y\rangle-V(y), & S&:=Q+t^{-1}\Id,\\
Q_{t}&:=Q(\Id+tQ)^{-1}, & m_{x}&:=t^{-1}S^{-1}x.
\end{aligned}
\]
The function $P$ is convex, and completing the square gives
\[
-\frac12\langle Qy,y\rangle-\frac{|x-y|^{2}}{2t}
=-\frac12\langle S(y-m_{x}),y-m_{x}\rangle
-\frac12\langle Q_{t}x,x\rangle,
\]
because \(t^{-1}\Id-t^{-2}S^{-1} =t^{-1}\bigl(\Id-(\Id+tQ)^{-1}\bigr) =Q_{t}\). After we substitute $y=z+m_{x}$, we obtain
\[
p_{t}(x)=c_{t}e^{-\langle Q_{t}x,x\rangle/2}h_{t}(x),
\qquad
h_{t}(x):=\int_{\R^{d}}e^{P(z+m_{x})}e^{-\langle Sz,z\rangle/2}\dd z,
\]
where $c_{t}:=Z_{\mu}^{-1}(2\pi t)^{-d/2}$.

For each fixed $z$, the map $x\mapsto e^{P(z+m_{x})}$ is log-convex because $m_{x}$ depends linearly on $x$. If $x,x'\in\R^{d}$ and $\lambda\in(0,1)$, convexity of $P$ and H\"older's inequality give \(h_{t}(\lambda x+(1-\lambda)x') \leq h_{t}(x)^{\lambda}h_{t}(x')^{1-\lambda}\). We conclude that $h_{t}$ is positive, finite, and log-convex. In particular, \(V_{t}(x)=\frac12\langle Q_{t}x,x\rangle-\log h_{t}(x)-\log c_{t}\) satisfies $\mathrm{D}^{2}V_{t}\preceq Q_{t}$. Since $Q$ commutes with $(\Id+tQ)^{-1}\preceq\Id$, we also have $Q_{t}\preceq Q$.

For the last statements, we couple $X$ and $X+\sqrt tG$, where $X\sim\mu$, $G\sim\gamma_{d}$, and the two variables are independent. This coupling gives the finite second moment and the estimate $W_{2}(\mu_{t},\mu)\leq\sqrt{td}$.
\end{proof}

\subsection{Proof of Theorem \ref{thm:support-curvature}}

We now have all the auxiliary results in place. Let us now consider the proof of the support--curvature theorem.

\begin{proof}[Proof of Theorem \ref{thm:support-curvature}]
By Lemma \ref{lem:range-uniqueness}, the potential is unique up to a constant, so we may normalise it by $\Phi(0)=0$. The proof has two steps. We first reduce the anisotropic case to the isotropic one by a linear change of variables, and then remove the smoothness of $V$ by Gaussian smoothing.

Suppose first that $V$ is smooth and $\mathrm{D}^{2}V\preceq Q$. Let $L:=Q^{1/2}$, which we identify with the linear map $x\mapsto Q^{1/2}x$, and let us define \(\widetilde\Phi:=\Phi\circ L^{-1}\). It follows that $\widetilde\Phi$ is convex, and \(\nabla\widetilde\Phi=L^{-1}\circ\nabla\Phi\circ L^{-1}\) pushes $\widetilde\mu:=L_{\#}\mu$ to $\widetilde\nu:=(L^{-1})_{\#}\nu$. The source potential $\widetilde V:=V\circ L^{-1}$ satisfies \(\mathrm{D}^{2}\widetilde V =L^{-1}(\mathrm{D}^{2}V\circ L^{-1})L^{-1} \preceq\Id\), and $\widetilde\nu$ is log-concave with support $L^{-1}K$.

We fix $v\neq0$ and set \(\widetilde e:=\frac{Lv}{|Lv|}\). For $x'\in\R^{d}$ and $h'\in\R$, set $x=L^{-1}x'$ and $h=h'|Lv|$. Then
\[
\widetilde\Phi(x'+h\widetilde e)+\widetilde\Phi(x'-h\widetilde e)-2\widetilde\Phi(x')
=\Phi(x+h'v)+\Phi(x-h'v)-2\Phi(x),
\]
and \(w_{L^{-1}K}(\widetilde e)=\frac{w_{K}(v)}{|Lv|}\). Proposition \ref{prop:smooth-sources}, which we apply to $\widetilde\Phi$ with $\Lambda=1$ in the direction $\widetilde e$, gives
\[
\Phi(x+h'v)+\Phi(x-h'v)-2\Phi(x)
\leq C_{1}|Lv|w_{K}(v)h'^{2}
=C_{1}\sqrt{\langle Qv,v\rangle}\,w_{K}(v)h'^{2}.
\]

We now take an arbitrary finite potential $V$, and set $\mu_{t}:=\mu*\gamma_{d,t}$. Let $\Phi_{t}$ be the convex function such that $(\nabla\Phi_{t})_{\#}\mu_{t}=\nu$ and $\Phi_{t}(0)=0$. By Lemma \ref{lem:source-smoothing}, the smooth potential $V_{t}$ satisfies $\mathrm{D}^{2}V_{t}\preceq Q$, so the previous argument gives \(\partial_{vv}\Phi_{t} \leq C_{1}\sqrt{\langle Qv,v\rangle}\,w_{K}(v)\) in second-difference form, uniformly in $t$. Since $\mu_{t}\to\mu$ in $W_{2}$, Lemma \ref{lem:stability} passes the estimate to $\Phi$, which proves \eqref{eq:support-curvature}.

Finally, \(\sqrt{\langle Qv,v\rangle}\,w_{K}(v) \leq\sqrt{\|Q\|_{\op}}\,\diam(K)|v|^{2}\), so Lemma \ref{lem:lipschitz-gradients} gives \eqref{eq:global-support-curvature}, and Lemma \ref{lem:range-uniqueness} yields $\nabla\Phi(\R^{d})\subset K$.
\end{proof}

\subsection{Semi-log-concave targets}

In the semi-log-concave setting, the perturbation $\delta|y|^{2}/2$ used in Proposition \ref{prop:remove-hessian} does not create a log-concave target, so we need a different argument for the a priori estimate \textup{(S3)}. We obtain it for targets whose potential is strongly convex near the boundary of a strictly convex support. The trick is to apply the fact that second differences vanish at infinity. For ball targets, this idea goes back to Caffarelli \cite{Caffarelli2000} (see also \cite[Lemma 3.1]{ColomboFigalliJhaveri2017}).

\begin{lemma}[Behaviour at infinity]\label{lem:infinity}
Let $\Omega\subset\R^{d}$ be bounded, open, and convex, and assume that $K:=\overline\Omega$ is strictly convex. Let $\Phi\in C^{1}(\R^{d})$ be convex and suppose that $T:=\nabla\Phi:\R^{d}\to\Omega$ is a homeomorphism. It follows that
\[
\operatorname{dist}(T(x),\partial K)\longrightarrow0
\qquad\text{as }|x|\to\infty.
\]
Moreover, for every $t>0$,
\begin{equation}\label{eq:infinity}
\sup_{e\in\Sd}\bigl[\Phi(x+te)+\Phi(x-te)-2\Phi(x)\bigr]
\longrightarrow0
\qquad\text{as }|x|\to\infty.
\end{equation}
\end{lemma}

\begin{proof}
The intuition is that, far away in a direction $\vartheta\in\Sd$, the map $T$ is close to the boundary point of $K$ with outer normal $\vartheta$, and this point depends continuously on $\vartheta$. We denote by \(h_{K}(\vartheta):=\sup_{y\in K}\langle y,\vartheta\rangle\) the support function of $K$. For $\vartheta\in\Sd$, let $s(\vartheta)$ be the unique point of $K$ at which the supremum is attained. By strict convexity and compactness, $s$ is continuous, and the caps \(\{y\in K:h_{K}(\vartheta)-\langle y,\vartheta\rangle\leq\delta\}\) have diameters which tend to zero as $\delta\downarrow0$, uniformly in $\vartheta$.

Let us fix $y_{0}\in\Omega$ and $\alpha\in(0,1)$, and define
\[
y_{\alpha}(\vartheta):=(1-\alpha)s(\vartheta)+\alpha y_{0}\in\Omega,
\qquad
z_{\alpha}(\vartheta):=T^{-1}(y_{\alpha}(\vartheta)).
\]
The inverse map is continuous, and the image of $y_{\alpha}$ is a compact subset of $\Omega$. This implies that \(M_{\alpha}:=\sup_{\vartheta\in\Sd}|z_{\alpha}(\vartheta)|<\infty\). For $x=R\vartheta$ with $R>0$, monotonicity of $T$ gives \(\langle T(x)-y_{\alpha}(\vartheta),x-z_{\alpha}(\vartheta)\rangle\geq0\). With $R_{K}:=\sup_{K}|y|$, we infer that
\[
h_{K}(\vartheta)-\langle T(R\vartheta),\vartheta\rangle
\leq \alpha\diam(K)+\frac{2R_{K}M_{\alpha}}{R}.
\]
Let us first choose $\alpha$ small and then $R$ large. By the cap property, \(T(R\vartheta)-s(\vartheta)\longrightarrow0\) uniformly in $\vartheta$, which proves the first statement.

For the second one, note that, as $|x|\to\infty$, the directions of $x\pm te$ approach $x/|x|$ uniformly in $e$. By uniform continuity of $s$, we obtain \(\sup_{e\in\Sd}|T(x+te)-T(x-te)|\longrightarrow0\). For $f(\tau):=\Phi(x+\tau e)$, convexity gives \(f(t)+f(-t)-2f(0)\leq t\bigl(f'(t)-f'(-t)\bigr)\), and \eqref{eq:infinity} follows.
\end{proof}

\begin{proposition}[A qualitative Hessian bound]\label{prop:qualitative-hessian}
Assume \textup{(S1)} and \textup{(S2)}, and suppose that
\[
\Omega=\{\psi<0\},
\qquad
K=\{\psi\leq0\},
\]
where $\psi$ is convex and smooth on a neighbourhood of $K$, and assume that $K$ is strictly convex. Suppose further that there exist $\delta,\kappa>0$ and a function $W^{\mathrm{col}}$, smooth on a neighbourhood of $K$ with
\[
\mathrm{D}^{2}W^{\mathrm{col}}\succeq\kappa\Id,
\]
such that
\[
W=W^{\mathrm{col}}
\qquad\text{on }K\cap\{\psi\geq-\delta\}.
\]
In this case, $M_{\Phi}<\infty$.
\end{proposition}

\begin{proof}
Let us explain the idea first. If the second differences of $\Phi$ were unbounded, their maxima would escape to infinity. There, the map $T$ takes values in the collar where the target potential is strongly convex, and the Monge--Amp\`ere equation turns this strong convexity into a bound on the second differences, as in Caffarelli's argument.

We set
\[
h(x,e):=\Phi_{ee}(x),
\qquad
h_{t}(x,e):=t^{-2}\bigl[\Phi(x+te)+\Phi(x-te)-2\Phi(x)\bigr]
\qquad (x\in\R^{d},\ e\in\Sd,\ t>0).
\]
By Proposition \ref{prop:regularity} and Lemma \ref{lem:infinity}, each $h_{t}$ is continuous, positive, and tends to zero as $|x|\to\infty$ uniformly in $e$, so a maximum is attained. Since $\Phi$ is smooth, $h_{t}\to h$ locally uniformly as $t\downarrow0$.

Now suppose that $\sup h=\infty$. We choose $(x_{m},e_{m})$ such that $h(x_{m},e_{m})\geq m$, and $t_{m}\downarrow0$ with $h_{t_{m}}(x_{m},e_{m})\geq m/2$. Let $(\widehat x_{m},\widehat e_{m})$ be a maximum point of $h_{t_{m}}$. We then have
\[
h_{t_{m}}(\widehat x_{m},\widehat e_{m})\longrightarrow\infty,
\qquad
|\widehat x_{m}|\longrightarrow\infty,
\]
because $h$ is locally bounded and $h_{t}\to h$ locally uniformly. Let us write
\[
\begin{aligned}
x_{\pm}&:=\widehat x_{m}\pm t_{m}\widehat e_{m}, & T_{0}&:=T(\widehat x_{m}), & T_{\pm}&:=T(x_{\pm}),\\
A_{0}&:=A(\widehat x_{m}), & A_{\pm}&:=A(x_{\pm}).
\end{aligned}
\]
The function \(x\longmapsto\Phi(x+t_{m}\widehat e_{m})+\Phi(x-t_{m}\widehat e_{m})-2\Phi(x)\) has a maximum at $\widehat x_{m}$, and so \(T_{+}+T_{-}=2T_{0}, \qquad A_{+}+A_{-}\preceq2A_{0}\).
By Lemma \ref{lem:infinity}, $\psi(T_{0})>-\delta/2$ for all large $m$. Since $\psi(T_{\pm})<0$ and $\psi$ is convex, we also have \(\psi(T_{+})+\psi(T_{-})\geq2\psi(T_{0})>-\delta\). We infer that $\psi(T_{\pm})>-\delta$, so $W=W^{\mathrm{col}}$ at $T_{0}$ and $T_{\pm}$. By concavity and monotonicity of $\log\det$ on positive definite matrices, \(\log\det A_{+}+\log\det A_{-}\leq2\log\det A_{0}\). Taken together with \eqref{eq:MA}, this yields
\[
W^{\mathrm{col}}(T_{+})+W^{\mathrm{col}}(T_{-})-2W^{\mathrm{col}}(T_{0})
\leq V(x_{+})+V(x_{-})-2V(\widehat x_{m})
\leq\Lambda t_{m}^{2}.
\]
Given that $T_{0}$ is the midpoint of $T_{+}$ and $T_{-}$ and $\mathrm{D}^{2}W^{\mathrm{col}}\succeq\kappa\Id$, the left-hand side is at least \(\frac{\kappa}{4}|T_{+}-T_{-}|^{2}\). It follows that \(|T_{+}-T_{-}|\leq2\sqrt{\Lambda/\kappa}\,t_{m}\), and convexity along the line $\widehat x_{m}+\R\widehat e_{m}$ gives
\[
\Phi(x_{+})+\Phi(x_{-})-2\Phi(\widehat x_{m})
\leq t_{m}\langle T_{+}-T_{-},\widehat e_{m}\rangle
\leq2\sqrt{\Lambda/\kappa}\,t_{m}^{2}.
\]
We deduce that \(h_{t_{m}}(\widehat x_{m},\widehat e_{m}) \leq2\sqrt{\Lambda/\kappa}\) for all large $m$, which is a contradiction.
\end{proof}

\begin{lemma}[Approximation with a strongly convex collar]\label{lem:collar-approximation}
Let
\[
\nu=Z_{\nu}^{-1}e^{-W}\1_{\Omega}\dd y
\]
be a $\rho$-semi-log-concave probability measure whose support is a convex body $K$, where $\Omega:=\operatorname{int}K$. There exist bounded open convex sets $\Omega_{n}\Subset\Omega$ with smooth uniformly convex boundaries, $K_{n}:=\overline{\Omega_{n}}$ with $\diam(K_{n})\leq\diam(K)$, convex functions $\psi_{n}$ smooth on a neighbourhood of $K_{n}$ such that
\[
\Omega_{n}=\{\psi_{n}<0\},
\qquad
K_{n}=\{\psi_{n}\leq0\},
\]
functions $W_{n}$ smooth on a neighbourhood of $K_{n}$ with $\mathrm{D}^{2}W_{n}\succeq-\rho\Id$, numbers $\delta_{n},\kappa_{n}>0$, and functions $W_{n}^{\mathrm{col}}$ smooth on a neighbourhood of $K_{n}$ such that
\[
\mathrm{D}^{2}W_{n}^{\mathrm{col}}\succeq\kappa_{n}\Id,
\qquad
W_{n}=W_{n}^{\mathrm{col}}
\quad\text{on }K_{n}\cap\{\psi_{n}\geq-\delta_{n}\}.
\]
Moreover,
\[
\nu_{n}:=Z_{n}^{-1}e^{-W_{n}}\1_{\Omega_{n}}\dd y
\longrightarrow\nu
\]
in total variation.
\end{lemma}

\begin{proof}
We combine the classical smoothing of convex functions with maximum gluing. The construction glues the mollified potential with a large multiple of the defining function of the approximating body. This multiple dominates near the boundary and is, in this manner, ``responsible'' for the strongly convex collar.

After a translation, we may assume that $0\in\Omega$. We define the Minkowski functional of $K$ by \(m(y):=\inf\{\lambda>0:y\in\lambda K\}, \qquad y\in\R^{d}\). Since $0\in\operatorname{int}K$, the function $m$ is finite, convex, and globally Lipschitz, and \(\Omega=\{m<1\}\). We choose a standard radial mollifier $\eta_{\eps}$ and set \(q_{\eps}(y):=(m*\eta_{\eps})(y)+\eps|y|^{2}\). Given that the mollifier has barycentre zero and $m$ is convex, Jensen's inequality gives $q_{\eps}\geq m$. Moreover, $q_{\eps}(0)\to0$ as $\eps\downarrow0$, and regular values are dense by Sard's theorem. In this situation, we may choose $\eps_{n}\downarrow0$ and regular values \(c_{n}\in\bigl(\max\{1-1/n,q_{\eps_{n}}(0)\},1\bigr)\) of $q_{\eps_{n}}$. We set:
\[
\psi_{n}:=q_{\eps_{n}}-c_{n},
\qquad
\Omega_{n}:=\{\psi_{n}<0\},
\qquad
K_{n}:=\{\psi_{n}\leq0\}.
\]
It follows that $K_{n}\subset\{m<1\}=\Omega$. Notice that the boundary of $\Omega_{n}$ is smooth and uniformly convex because \(\mathrm{D}^{2}q_{\eps_{n}}\succeq2\eps_{n}\Id\), and since $K_{n}\subset K$, its diameter is at most $\diam(K)$.

We now verify that the sets $\Omega_{n}$ exhaust $\Omega$. On a compact subset of $\Omega$, we have $m\leq1-2\beta$ for some $\beta>0$. Since $q_{\eps_{n}}\to m$ uniformly there and $c_{n}\to1$, this compact set lies in $\Omega_{n}$ for all sufficiently large $n$.

Let us now construct the potentials. The function \(G:=W+\frac{\rho}{2}|y|^{2}\) is convex on $\Omega$, and mollification at a scale smaller than $\operatorname{dist}(K_{n},\partial\Omega)$ gives a convex function $G_{n}^{0}$, smooth on a neighbourhood of $K_{n}$, such that \(\sup_{K_{n}}|G_{n}^{0}-G|\leq\frac1n\). We choose $\delta_{n}>0$ so small that \(C_{n}:=\{\psi_{n}\leq-3\delta_{n}\}\) satisfies $\nu(C_{n})\geq\nu(\Omega_{n})-1/n$, and set \(O_{n}:=K_{n}\cap\{\psi_{n}\geq-2\delta_{n}\}\). In particular, $C_{n}$ carries almost all of the mass of $\Omega_{n}$, and $O_{n}$ is a thin layer near the boundary in which we place the collar. We also fix a smooth even convex function $\chi:\R\to\R$ such that $\chi(s)=|s|$ for $|s|\geq2$, and define \(\operatorname{smax}(a,b):=\frac12\bigl(a+b+\chi(a-b)\bigr)\). This smooth maximum is convex and nondecreasing in each variable, dominates $\max\{a,b\}$, and equals $\max\{a,b\}$ when $|a-b|\geq2$ \cite[Lemma 2.1]{Azagra2013}.

With \(\operatorname{osc}_{K_{n}}G_{n}^{0}:=\sup_{K_{n}}G_{n}^{0}-\inf_{K_{n}}G_{n}^{0}\), we choose $M_{n}>0$ such that
\[
M_{n}\delta_{n}>\operatorname{osc}_{K_{n}}G_{n}^{0}+4,
\qquad
2\eps_{n}M_{n}\geq\rho+1,
\]
and $\tau_{n}\in\R$ such that \(\overline G_{n}:=M_{n}\psi_{n}+\tau_{n}\) satisfies
\[
\overline G_{n}\leq G_{n}^{0}-2\quad\text{on }C_{n},
\qquad
\overline G_{n}\geq G_{n}^{0}+2\quad\text{on }O_{n}.
\]
Such a number exists because $-M_{n}\psi_{n}\leq2M_{n}\delta_{n}$ on $O_{n}$ and $-M_{n}\psi_{n}\geq3M_{n}\delta_{n}$ on $C_{n}$, and the first condition on $M_{n}$ makes the admissible interval for $\tau_{n}$ nonempty.

We define
\[
G_{n}:=\operatorname{smax}(G_{n}^{0},\overline G_{n}),
\qquad
W_{n}:=G_{n}-\frac{\rho}{2}|y|^{2},
\qquad
W_{n}^{\mathrm{col}}:=\overline G_{n}-\frac{\rho}{2}|y|^{2},
\]
and let \(\kappa_{n}:=2\eps_{n}M_{n}-\rho\geq1\). It follows that $G_{n}$ is smooth and convex on a neighbourhood of $K_{n}$, so $\mathrm{D}^{2}W_{n}\succeq-\rho\Id$. By the choice of $\tau_{n}$ and the properties of the smooth maximum, we have $G_{n}=G_{n}^{0}$ on $C_{n}$ and $G_{n}=\overline G_{n}$ on $O_{n}$. Moreover, \(\mathrm{D}^{2}W_{n}^{\mathrm{col}} =M_{n}\mathrm{D}^{2}\psi_{n}-\rho\Id \succeq\kappa_{n}\Id\), and $W_{n}=W_{n}^{\mathrm{col}}$ on \(O_{n}=K_{n}\cap\{\psi_{n}\geq-2\delta_{n}\}\). This verifies the condition with $2\delta_{n}$ in place of $\delta_{n}$.

It remains to prove convergence. To this end, let
\[
f:=e^{-W}\1_{\Omega},
\qquad
f_{n}^{0}:=e^{-G_{n}^{0}+\rho|y|^{2}/2}\1_{\Omega_{n}},
\qquad
f_{n}:=e^{-W_{n}}\1_{\Omega_{n}}.
\]
We have
\[
e^{-1/n}f\leq f_{n}^{0}\leq e^{1/n}f
\quad\text{on }\Omega_{n},
\qquad
0\leq f_{n}\leq f_{n}^{0},
\qquad
f_{n}=f_{n}^{0}\quad\text{on }C_{n},
\]
and it follows that
\begin{align*}
\|f_{n}-f\|_{L^{1}}
&\leq(e^{1/n}-1)\int f
+\int_{\Omega_{n}\setminus C_{n}}f_{n}^{0}
+\int_{\Omega\setminus\Omega_{n}}f\\
&\leq(e^{1/n}-1)Z_{\nu}
+e^{1/n}Z_{\nu}\nu(\Omega_{n}\setminus C_{n})
+Z_{\nu}\nu(\Omega\setminus\Omega_{n})
\longrightarrow0.
\end{align*}
We conclude that $\nu_{n}\to\nu$ in total variation.
\end{proof}

\begin{proof}[Proof of Theorem \ref{thm:semilogconcave}]
We normalise $\Phi(0)=0$. The strategy is the same as for Theorem \ref{thm:support-curvature}, except that the qualitative bound \textup{(S3)} now comes from Proposition \ref{prop:qualitative-hessian}. Suppose first that $V$ is smooth and $\mathrm{D}^{2}V\preceq\Lambda\Id$, so that \textup{(S1)} holds. Let $\nu_{n}$, $\Omega_{n}$, $K_{n}$, $\psi_{n}$, $W_{n}$, $\delta_{n}$, $\kappa_{n}$, and $W_{n}^{\mathrm{col}}$ be given by Lemma \ref{lem:collar-approximation}, and let $\Phi_{n}$ be the convex function such that $(\nabla\Phi_{n})_{\#}\mu=\nu_{n}$ and $\Phi_{n}(0)=0$. It follows that \textup{(S2)} holds for $\nu_{n}$, and Proposition \ref{prop:qualitative-hessian} gives $M_{\Phi_{n}}<\infty$. Since \(\rho\diam(K_{n})^{2}\leq r\) and $C_{2}$ is nondecreasing by Proposition \ref{prop:C2}, we may apply Proposition \ref{prop:semilog-estimate} to obtain
\[
\Phi_{n}(x+h)+\Phi_{n}(x-h)-2\Phi_{n}(x)
\leq C_{2}(r)\sqrt\Lambda\,\diam(K)|h|^{2}
\qquad (x,h\in\R^{d}).
\]
The supports are in the fixed compact convex set $K$, and $\nu_{n}\rightharpoonup\nu$. Therefore, Lemma \ref{lem:stability} passes the second-difference estimate to $\Phi$, one fixed vector $h$ at a time, and Lemma \ref{lem:lipschitz-gradients} gives the Lipschitz conclusion.

Let us now take a general finite potential $V$. We repeat the Gaussian source smoothing and the stability argument from the proof of Theorem \ref{thm:support-curvature}, with $Q=\Lambda\Id$ and the uniform estimate obtained above. Lemma \ref{lem:source-smoothing} gives smooth source potentials $V_{t}$ with $\mathrm{D}^{2}V_{t}\preceq\Lambda\Id$ and $\mu_{t}\to\mu$ in the Wasserstein distance $W_{2}$. Since the smooth-source estimate is uniform in $t$, Lemma \ref{lem:stability} passes it to the locally uniform limit in second-difference form, and a final application of Lemma \ref{lem:lipschitz-gradients} proves the theorem.
\end{proof}

\section{Sharp diameter bounds}\label{sec:sharp-diameter}

In this section, we prove the diameter estimates. The inverse transport admits a two-point comparison which gives the exact bound for an isotropic Gaussian source, and the same comparison gives the exponential rate $1/8$ when the source potential is convex. We work with the inverse map because a Lipschitz bound for the forward map is the same as a lower bound for the Hessian of the inverse potential, and such a lower bound can be obtained from the minimum of a two-point function.

\subsection{An inverse two-point comparison}

\begin{lemma}[Inverse reflection estimate]\label{lem:inverse-reflection}
Let us assume \textup{(S1)}--\textup{(S2)}, suppose that $K$ is strictly convex, and assume that $\mathrm{D}^{2}W\succeq-\rho\Id$, where $\rho\in\R$. We write
\[
D:=\diam(K),
\qquad
\Psi:=\Phi^{*},
\qquad
S:=\nabla\Psi=T^{-1},
\qquad
B:=\mathrm{D}^{2}\Psi.
\]
Let $R>D/2$, and let $f\in C^{3}([0,R))$ satisfy
\[
f(0)=0,
\qquad
f'>0,
\qquad
f''\geq0.
\]
For distinct $x,y\in\Omega$, we define
\[
\begin{aligned}
2s&:=|y-x|, & e&:=\frac{y-x}{2s},\\
\delta&:=S(y)-S(x), & F(x,y)&:=\langle\delta,e\rangle-2f(s).
\end{aligned}
\]
If $F$ takes a negative value, then it has a negative minimum at an interior pair of distinct points. At such a pair, if we set
\[
a:=f'(s),
\qquad
\delta_{e}:=\langle\delta,e\rangle,
\qquad
\delta_{\perp}:=\delta-\langle\delta,e\rangle e,
\]
then
\begin{equation}\label{eq:inverse-first-order}
B(x)e=B(y)e=b:=ae-\frac{\delta_{\perp}}{2s}
\end{equation}
and
\begin{align}
\frac{2f''(s)}{a}
&\leq-\bigl(W_{e}(y)-W_{e}(x)\bigr)
+\left\langle b,\nabla V(S(y))-\nabla V(S(x))\right\rangle \notag\\
&\leq2\rho s
+\left\langle b,\nabla V(S(y))-\nabla V(S(x))\right\rangle.
\label{eq:inverse-second-order}
\end{align}
If $F\geq0$ everywhere, then
\[
B\succeq f'(0)\Id
\qquad\text{on }\Omega,
\]
and
\[
\Lip(T)\leq\frac1{f'(0)}.
\]
\end{lemma}

Since $2s=|y-x|$ and $e=(y-x)/|y-x|$, the inequality $F\geq0$ for all distinct $x,y\in\Omega$ means, in the terminology of Andrews and Clutterbuck \cite{AndrewsClutterbuck2011}, that $f$ is a modulus of convexity for $\Psi$. In their setting, the second-order part of the operator is the Laplacian at both points of the pair. In our case, the linearisation of the inverse Monge--Amp\`ere equation \eqref{eq:inverse-MA} has the coefficient matrices $B(x)^{-1}$ and $B(y)^{-1}$ at the two points, which differ in general, so we adapt the reflection coupling in the proof below to these two matrices.

\begin{proof}
By Proposition \ref{prop:regularity}, the map $T:\R^{d}\to\Omega$ is a smooth diffeomorphism and $B\succ0$. Moreover, the inverse Monge--Amp\`ere equation reads
\begin{equation}\label{eq:inverse-MA}
\log\det B=-W+V(S)+c
\end{equation}
for a constant $c$. The proof has three steps. We first show that a negative infimum of $F$ is attained in the interior, we then derive the first-order conditions, and finally, we test the second-order condition along a random perturbation of both endpoints. We choose this perturbation so that the expected second variation can be computed from the inverse Monge--Amp\`ere equation.

\textit{Localisation.}
Let us first localise a negative minimum. Monotonicity of $S$ gives \(\delta_{e}:=\langle\delta,e\rangle\geq0\), so $F\geq-2f(s)$, and the lower limit of $F$ is nonnegative along every sequence which approaches the diagonal, at the boundary as well.

Next, we study the behaviour near the boundary. Suppose that $\Omega\ni w_{n}\to w\in\partial K$. We show that $|S(w_{n})|\to\infty$. Indeed, otherwise, a convergent subsequence of $S(w_{n})$, with continuity of $T$, would give $w\in\Omega$. We also claim that every limit direction \(\vartheta=\lim_{k\to\infty}\frac{S(w_{n_{k}})}{|S(w_{n_{k}})|}\) with $n_{k}\to\infty$ is an outer unit normal at $w$. To see this, observe that monotonicity gives \(\langle w_{n}-q,S(w_{n})-S(q)\rangle\geq0 \qquad (q\in\Omega)\), and division by $|S(w_{n})|$ yields $\langle w-q,\vartheta\rangle\geq0$ in the limit. Strict convexity implies \(\langle w-q,\vartheta\rangle>0 \qquad (q\in K\setminus\{w\})\), because equality would put the nontrivial segment $[w,q]$ in a supporting hyperplane of $K$.

We infer that, if $\Omega\ni x_{n}\to x$ and $\Omega\ni y_{n}\to y$ with $x\neq y$, then a boundary endpoint $y$ contributes a term which tends to $+\infty$ to $\langle S(y_{n}),\widehat e_{n}\rangle$, while a boundary endpoint $x$ contributes a term which tends to $+\infty$ to $-\langle S(x_{n}),\widehat e_{n}\rangle$, where \(\widehat e_{n}:=\frac{y_{n}-x_{n}}{|y_{n}-x_{n}|}\). An interior endpoint contributes a bounded term, and these conclusions hold for the full sequence, as we may pass to subsequences of the unit directions. Since $f$ is bounded on $[0,D/2]$, we obtain $F(x_{n},y_{n})\to+\infty$ whenever one endpoint lies on the boundary and the endpoint limits are distinct. Compactness of $K$, the diagonal estimate, and this boundary estimate show that a negative infimum is attained at an interior pair with $x\neq y$.

\textit{First-order conditions.}
Let us fix such a pair and choose coordinates with $e=e_{1}$. For $u\in\R^{d}$, we write $u_{e}:=\langle u,e\rangle$ and $u_{\perp}:=u-u_{e}e$, and identify vectors orthogonal to $e$ with their last $d-1$ coordinates. If we differentiate $F$ in the two endpoints, we obtain
\[
\nabla_{x}F=-B(x)e-\frac{\delta_{\perp}}{2s}+ae,
\qquad
\nabla_{y}F=B(y)e+\frac{\delta_{\perp}}{2s}-ae,
\]
and the first-order conditions give \eqref{eq:inverse-first-order}. In particular, $B(x)$ and $B(y)$ have the same first column, which makes the following block decomposition possible. We set \(\beta:=-\frac{\delta_{\perp}}{2s}, \qquad B_{0}:=B(x), \qquad B_{1}:=B(y)\), and write
\[
B_{i}=
\begin{pmatrix}
  a&\beta^{\top}\\
  \beta&C_{i}
\end{pmatrix},
\qquad
\widehat C_{i}:=C_{i}-\frac{\beta\beta^{\top}}{a}\succ0.
\]

\textit{The reflection coupling.}
We now construct the random perturbation. Its covariance at each endpoint is the inverse of the corresponding Hessian of $\Psi$, and its component along the chord is reflected, so that the two endpoints move in opposite directions along the chord. Let $\xi$ be a standard real Gaussian, and let $\zeta$ be an independent standard Gaussian in $\R^{d-1}$. We define \(z_{i}:=\widehat C_{i}^{-1/2}\zeta\) and
\[
v_{0}:=\left(\frac{\xi}{\sqrt a}-\frac{\beta\cdot z_{0}}{a},z_{0}\right),
\qquad
v_{1}:=\left(-\frac{\xi}{\sqrt a}-\frac{\beta\cdot z_{1}}{a},z_{1}\right).
\]
By the block inverse formula,
\[
\mathbb{E}[v_{i}v_{i}^{\top}]=B_{i}^{-1}
=\begin{pmatrix}
a^{-1}+a^{-2}\beta^{\top}\widehat C_{i}^{-1}\beta&-a^{-1}\beta^{\top}\widehat C_{i}^{-1}\\
-a^{-1}\widehat C_{i}^{-1}\beta&\widehat C_{i}^{-1}
\end{pmatrix}.
\]
We set \(\Delta:=v_{1}-v_{0}, \qquad z:=z_{1}-z_{0}=\Delta_{\perp}\). Since $\xi$ and $\zeta$ are independent, we have
\[
\mathbb{E}\Delta_{e}^{2}
=\frac4a+\frac{\mathbb{E}(\beta\cdot z)^{2}}{a^{2}},
\qquad
\mathbb{E}\bigl[\Delta_{e}(\beta\cdot z)\bigr]
=-\frac{\mathbb{E}(\beta\cdot z)^{2}}a,
\]
and
\[
\mathbb{E}\bigl[(B_{1}v_{1}-B_{0}v_{0})_{\perp}\cdot z\bigr]
=2(d-1)-\tr(\widehat C_{1}^{1/2}\widehat C_{0}^{-1/2})-\tr(\widehat C_{0}^{1/2}\widehat C_{1}^{-1/2})
\leq0.
\]
For the last identity, note that the transverse component of $B_{1}v_{1}-B_{0}v_{0}$ is \(-\frac{2\beta\xi}{\sqrt a}+(\widehat C_{1}^{1/2}-\widehat C_{0}^{1/2})\zeta\). The trace inequality does not require $\widehat C_{0}$ and $\widehat C_{1}$ to commute. Indeed, \(M:=\widehat C_{1}^{1/2}\widehat C_{0}^{-1/2}\) is similar to the positive definite matrix $\widehat C_{0}^{-1/4}\widehat C_{1}^{1/2}\widehat C_{0}^{-1/4}$, and $M^{-1}=\widehat C_{0}^{1/2}\widehat C_{1}^{-1/2}$. It follows that the positive eigenvalues $\lambda_{j}$ of $M$ satisfy \(\tr M+\tr M^{-1}=\sum_{j}(\lambda_{j}+\lambda_{j}^{-1})\geq2(d-1)\).

\textit{Second-order condition.}
We next compute the second variation. Along the curve \(\tau\longmapsto(x+\tau v_{0},y+\tau v_{1})\), we have at $\tau=0$,
\[
s'=\frac{\Delta_{e}}2,
\qquad
s''=\frac{|z|^{2}}{4s},
\qquad
e'=\frac{z}{2s},
\qquad
e''=-\frac{\Delta_{e}z}{2s^{2}}-\frac{|z|^{2}e}{4s^{2}}.
\]
If we use $\delta_{\perp}=-2s\beta$ in the second derivative of $\langle\delta,e\rangle-2f(s)$, we obtain
\begin{align}
\mathrm{D}^{2}F[(v_{0},v_{1}),(v_{0},v_{1})]
&=\mathrm{D}^{3}\Psi(y)[e,v_{1},v_{1}]-\mathrm{D}^{3}\Psi(x)[e,v_{0},v_{0}]\notag\\
&\quad+\frac1s(B_{1}v_{1}-B_{0}v_{0})_{\perp}\cdot z
+\frac1s\Delta_{e}(\beta\cdot z)\notag\\
&\quad-\left(\frac{\delta_{e}}{4s^{2}}+\frac{a}{2s}\right)|z|^{2}
-\frac{f''(s)}2\Delta_{e}^{2}.
\label{eq:inverse-second-variation}
\end{align}
The left-hand side is nonnegative at the minimum. We take expectations and use the preceding identities, with $\delta_{e}\geq0$ and $f''\geq0$, to obtain \(0\leq\mathbb{E}\mathrm{D}^{2}F[(v_{0},v_{1}),(v_{0},v_{1})] \leq L-\frac{2f''(s)}a\), where \(L:=\tr(B_{1}^{-1}(\partial_{e}B)(y))-\tr(B_{0}^{-1}(\partial_{e}B)(x))\). We note that, in expectation, the third derivatives of $\Psi$ appear only through the traces in $L$, which are controlled by the inverse Monge--Amp\`ere equation. The same calculation applies in dimension one, where the transverse blocks and terms are absent.

We now differentiate \eqref{eq:inverse-MA} in the direction $e$ at the two endpoints and use \eqref{eq:inverse-first-order}. This gives
\[
L=-\bigl(W_{e}(y)-W_{e}(x)\bigr)
+\left\langle b,\nabla V(S(y))-\nabla V(S(x))\right\rangle.
\]
Finally, \(W_{e}(y)-W_{e}(x)\geq-2\rho s\), which proves \eqref{eq:inverse-second-order}.

Now suppose that $F\geq0$ everywhere. We fix $q\in\Omega$ and $e\in\Sd$, take $x=q-he$ and $y=q+he$, and let $h\downarrow0$, which gives \(\langle B(q)e,e\rangle\geq f'(0)\). Since the direction was arbitrary, $B\succeq f'(0)\Id$ on $\Omega$. Given that \(\mathrm{D}^{2}\Phi=B(T)^{-1}\), we have $\mathrm{D}^{2}\Phi\preceq f'(0)^{-1}\Id$ on $\R^{d}$, which gives the Lipschitz estimate in question.
\end{proof}

\subsection{Approximation and passage to the limit}

\begin{lemma}[Preserving the curvature and diameter]\label{lem:preserve-curvature}
Let $K$ be a convex body, set $\Omega:=\operatorname{int}K$, let $\rho\in\R$, and assume that $W+\rho|\cdot|^{2}/2$ is finite and convex on $\Omega$, with
\[
\nu=Z_{\nu}^{-1}e^{-W}\1_{\Omega}\dd y
\]
a probability measure. There exist smooth strictly convex bodies $K_{n}\Subset\Omega$ and potentials $W_{n}$, smooth on neighbourhoods of $K_{n}$, such that
\[
\mathrm{D}^{2}W_{n}\succeq-\rho\Id,
\qquad
\diam(K_{n})\leq\diam(K),
\]
and
\[
\nu_{n}:=Z_{n}^{-1}e^{-W_{n}}\1_{K_{n}}\dd y\longrightarrow\nu
\]
in total variation and in the Wasserstein distance $W_{2}$.

If, in addition, $V$ and $\Lambda|\cdot|^{2}/2-V$ are finite and convex on $\R^{d}$ for some $\Lambda>0$ and $e^{-V}$ is integrable, then the source $\mu:=Z_{\mu}^{-1}e^{-V}\dd x$ can be approximated in the Wasserstein distance $W_{2}$ by smooth positive densities with potentials $V_{n}$ that satisfy
\[
0\preceq \mathrm{D}^{2}V_{n}\preceq\Lambda\Id.
\]
Finally, for either a fixed isotropic Gaussian source or these approximating sources, every uniform Lipschitz estimate for the corresponding Brenier maps passes to the limiting Brenier map.
\end{lemma}

\begin{proof}
The geometric construction in the proof of Lemma \ref{lem:collar-approximation} gives smooth uniformly convex bodies $K_{n}\Subset\Omega$ such that every compact subset of $\Omega$ lies in $K_{n}$ for all sufficiently large $n$. In particular, \(\diam(K_{n})\leq\diam(K)\). Note that this construction depends only on $K$, so it applies for every real $\rho$. We fix an even nonnegative smooth mollifier $\eta$ of integral one, supported in the unit ball, and write $\eta_{\eps}(x):=\eps^{-d}\eta(x/\eps)$.

For the target, we set \(U:=W+\frac{\rho}{2}|\cdot|^{2}\). The function $U$ is finite and convex on $\Omega$, so it is locally Lipschitz. We choose $\delta_{n}>0$ so small that $U*\eta_{\delta_{n}}$ is defined on a neighbourhood of $K_{n}$ and \(\sup_{K_{n}}|U*\eta_{\delta_{n}}-U|\leq\frac1n\). We then define \(W_{n}:=U*\eta_{\delta_{n}}-\frac{\rho}{2}|\cdot|^{2}\). The convolution is convex, so $\mathrm{D}^{2}(U*\eta_{\delta_{n}})\succeq0$, and hence $\mathrm{D}^{2}W_{n}\succeq-\rho\Id$ for either sign of $\rho$. The unnormalised densities $e^{-W_{n}}\1_{K_{n}}$ converge almost everywhere to $e^{-W}\1_{\Omega}$ and are bounded by $e^{1-W}\1_{\Omega}$. By dominated convergence, they converge in $L^{1}$, and the normalising constants converge as well. After normalisation, the densities converge in total variation, and since all measures are supported in the fixed compact set $K$, they also converge in the Wasserstein distance $W_{2}$.

For the source, we mollify the potential directly, that is, we set \(V_{\eps}:=V*\eta_{\eps}\) and \(V_{n}:=V_{\eps_{n}}\) for some $\eps_{n}\downarrow0$. Note that convolution preserves convexity of $V$ and of $\Lambda|\cdot|^{2}/2-V$, up to an irrelevant additive constant in the second function. As a result, \(0\preceq \mathrm{D}^{2}V_{\eps}\preceq\Lambda\Id\). There is also an interesting global error estimate. More precisely, the two convexity assumptions give \(0\leq V(x+h)+V(x-h)-2V(x)\leq\Lambda|h|^{2}\), and since $\eta$ is even,
\[
0\leq V_{\eps}(x)-V(x)
\leq\frac{\Lambda}{2}\int|h|^{2}\eta_{\eps}(h)\dd h
\leq\frac{\Lambda\eps^{2}}2.
\]

Recall that an integrable full-dimensional log-concave density has exponential tails. Indeed, a sublevel set of $V$ which contains a ball is bounded because an unbounded convex set with nonempty interior has infinite volume, and convexity along rays beyond such a sublevel set gives \(V(x)\geq c|x|-C\) for some $c>0$ and $C\in\R$. In particular, the original source has finite second moment, and the global error estimate proves convergence of the normalised source densities in \(L^{1}((1+|x|^{2})\dd x)\), and thus in the Wasserstein distance $W_{2}$.

Finally, suppose that the approximating Brenier maps have a common Lipschitz bound $L$, and let us normalise their potentials $\Phi_{n}$ at the origin. The source densities are positive, the sources converge in the Wasserstein distance $W_{2}$, and the targets converge weakly with supports in $K$. The Lipschitz estimate also gives \(\Phi_{n}(x+h)+\Phi_{n}(x-h)-2\Phi_{n}(x)\leq L|h|^{2}\). By Lemma \ref{lem:stability}, the potentials converge locally uniformly to the normalised limiting Brenier potential, and this second-difference bound passes to the limit. Lemma \ref{lem:lipschitz-gradients} then gives an everywhere-defined $L$-Lipschitz gradient.
\end{proof}

\subsection{The sharp isotropic Gaussian bound}

\begin{proof}[Proof of Theorem \ref{thm:sharp-gaussian}]
Let us first consider the standard Gaussian source. By Lemma \ref{lem:preserve-curvature}, it suffices to prove the estimate when $K$ is smooth and strictly convex and $W$ is smooth on a neighbourhood of $K$, with $\mathrm{D}^{2}W\succeq-\rho\Id$.

We choose the profile in Lemma \ref{lem:inverse-reflection} as the inverse of the one-dimensional transport onto a slightly enlarged extremal interval model. More precisely, we fix $R>D/2$, set \(Z_{\rho,R}:=\int_{-R}^{R}e^{\rho t^{2}/2}\dd t\), and let $f:(-R,R)\to\R$ be the increasing rearrangement from \(Z_{\rho,R}^{-1}e^{\rho s^{2}/2}\1_{(-R,R)}\dd s\) to $\gamma_{1}$. It is smooth, odd, and strictly increasing, and density balance gives
\begin{equation}\label{eq:profile-ODE}
\frac{f''(s)}{f'(s)}=\rho s+f(s)f'(s),
\qquad
f'(0)=\frac{\sqrt{2\pi}}{Z_{\rho,R}}.
\end{equation}
We first verify that $f''\geq0$ on $[0,R)$. For $\rho\geq0$, the first identity in \eqref{eq:profile-ODE} has a positive right-hand side for $s>0$. Suppose that $\rho=-\kappa<0$. In this case, \(Z_{-\kappa,R}<\sqrt{2\pi/\kappa}\), so $f'(0)>\sqrt\kappa$. We set \(q:=\frac{f''}{f'}=-\kappa s+ff'\) and observe that $q(0)=0$ and \(q'(0)=f'(0)^{2}-\kappa>0\). If $q$ had a first positive zero $s_{0}$ after this initial interval of positivity, then $f''>0$ on $(0,s_{0})$ and \(q'(s_{0})=-\kappa+f'(s_{0})^{2}+f(s_{0})f''(s_{0}) =f'(s_{0})^{2}-\kappa>0\), which contradicts $q'(s_{0})\leq0$ at a zero reached from positive values. We conclude that $f''>0$ on $(0,R)$, while $f''(0)=0$ by oddness.

We now apply Lemma \ref{lem:inverse-reflection} with this function $f$. At a hypothetical negative minimum, we have \(0\leq\delta_{e}<2f(s), \qquad \nabla V(S(y))-\nabla V(S(x))=\delta\), where the second identity holds because $\nabla V$ is the identity for the standard Gaussian. Therefore,
\begin{align*}
\frac{2f''(s)}{a}
&\leq2\rho s+\langle b,\delta\rangle\\
&=2\rho s+a\delta_{e}-\frac{|\delta_{\perp}|^{2}}{2s}\\
&<2\rho s+2af(s)
=\frac{2f''(s)}{a},
\end{align*}
where the final identity follows from \eqref{eq:profile-ODE}. This is impossible, so $F\geq0$, and Lemma \ref{lem:inverse-reflection} gives \(\Lip(T)\leq\frac{Z_{\rho,R}}{\sqrt{2\pi}}\). As $R\downarrow D/2$, we obtain the estimate in the smooth case. For the approximating targets, the right-hand side is bounded by the same expression with the diameter of the original target because the integral over $[-D/2,D/2]$ is increasing in $D$ for every real $\rho$. Finally, Lemma \ref{lem:preserve-curvature} proves the general case.

If the source has covariance $\Lambda^{-1}\Id$, the change of variables $z=\sqrt\Lambda x$ multiplies the Lipschitz constant by $\sqrt\Lambda$. Indeed, if $T_{1}=\nabla\Phi_{1}$ denotes the Brenier map from $\gamma_{d}$ to $\nu$, then
\[
T(x)=T_{1}(\sqrt\Lambda x)
=\nabla\left[\frac1{\sqrt\Lambda}\Phi_{1}(\sqrt\Lambda x)\right].
\]

Let us now prove sharpness, which includes the statement in every fixed dimension. We write
\[
L_{\rho}(\ell):=\frac1{\sqrt{2\pi}}
\int_{-\ell/2}^{\ell/2}e^{\rho t^{2}/2}\dd t.
\]
In dimension one, we take the target with density proportional to $e^{\rho y^{2}/2}$ on $[-D/2,D/2]$ and the Gaussian source of precision $\Lambda$. Both medians are zero, and density balance gives \(T'(0)=\sqrt\Lambda L_{\rho}(D)\). With the upper estimate we have already established, this shows that equality holds for the Lipschitz constant.

We now fix $d\geq2$ and approximate the one-dimensional extremal target by thin boxes with diameter exactly $D$. For \(0<\eps<\frac{D}{\sqrt{d-1}}\), we set
\[
\ell_{\eps}:=\sqrt{D^{2}-(d-1)\eps^{2}},
\qquad
K_{\eps}:=[-\ell_{\eps}/2,\ell_{\eps}/2]
\times[-\eps/2,\eps/2]^{d-1}.
\]
Observe that this full-dimensional box has diameter exactly $D$. The target density proportional to $e^{\rho|y|^{2}/2}$ on $K_{\eps}$ has potential Hessian $-\rho\Id$ and is a product measure. The product of the increasing one-dimensional maps is a convex gradient which pushes the product Gaussian to this measure, so it is the Brenier map. Its first diagonal derivative at the origin is \(\sqrt\Lambda L_{\rho}(\ell_{\eps})\), which converges to $\sqrt\Lambda L_{\rho}(D)$. This proves the exact supremum in every fixed dimension and for every real $\rho$.

For later proofs, we give here the following asymptotics of the coefficient $\Bp$ in \eqref{eq:B-def}:
\begin{equation}\label{eq:B-asymptotic}
\frac1{\sqrt{2\pi}}\int_{-1/2}^{1/2}e^{rv^{2}/2}\dd v
\sim\frac{4}{\sqrt{2\pi}\,r}e^{r/8}
\qquad (r\to\infty).
\end{equation}
Indeed, symmetry and the substitution $v=1/2-u/r$ give
\[
\int_{-1/2}^{1/2}e^{rv^{2}/2}\dd v
=\frac{2e^{r/8}}{r}\int_{0}^{r/2}e^{-u/2+u^{2}/(2r)}\dd u.
\]
After extension by zero, the integrand is bounded by $e^{-u/4}$ and converges to $e^{-u/2}$, whose integral is $2$, so the claim follows immediately by dominated convergence.
\end{proof}

\begin{proof}[Proof of Corollary \ref{cor:gaussian-source}]
The directional statement is the special case $Q=\Id$ of Theorem \ref{thm:support-curvature}. For a full-dimensional target, the sharp diameter estimate follows from Theorem \ref{thm:sharp-gaussian} with $\rho=0$ and $\Lambda=1$.

Let us now consider a target supported on an affine subspace. Since the Gaussian factorises along the linear subspace parallel to the affine hull of the target and along the orthogonal complement of that linear subspace, the problem reduces to a full-dimensional one on that linear subspace. We write the affine hull as $a+E$, where $E$ is a linear subspace and $a\perp E$, and let $P_{E}$ be orthogonal projection onto $E$. If $k:=\dim E\geq1$, the translated measure \(\widetilde\nu:=(y\mapsto y-a)_{\#}\nu\) is a full-dimensional log-concave probability measure on $E$, and $(P_{E})_{\#}\gamma_{d}=\gamma_{E}$ is the standard Gaussian on this space. Let \(\widetilde T=\nabla_{E}\widetilde\Phi\) be the Brenier map between these measures, where $\nabla_{E}$ denotes the gradient on $E$. It follows that
\[
T_{0}(x):=a+\widetilde T(P_{E}x)
=\nabla\bigl(\langle a,x\rangle+\widetilde\Phi(P_{E}x)\bigr)
\]
pushes $\gamma_{d}$ to $\nu$, and its potential is convex, so Brenier uniqueness identifies $T_{0}$ with $T$. Since the reduced target has the same diameter, we obtain \(\Lip(T)\leq\Lip(\widetilde T)\leq\frac{\diam(K)}{\sqrt{2\pi}}\). If $k=0$, the target is a point mass and $T$ is constant. Finally, the product examples in the proof of Theorem \ref{thm:sharp-gaussian} show that the diameter coefficient is sharp even among full-dimensional targets in every fixed dimension.
\end{proof}

\begin{remark}[Targets on affine subspaces]\label{rem:affine}
The reduction in the proof of Corollary \ref{cor:gaussian-source} also applies to Theorem \ref{thm:sharp-gaussian} for every $\rho\in\R$. Let $\nu$ be supported on $a+E$, where $E$ is a linear subspace, $a\perp E$, and $\dim E\geq1$, and suppose that $\nu$ has density $Z_{\nu}^{-1}e^{-W}$ with respect to Lebesgue measure on $a+E$, where $W+\frac{\rho}{2}|\cdot|^{2}$ is convex on the relative interior of the support. Because $a\perp E$, we have $|a+z|^{2}=|a|^{2}+|z|^{2}$ for $z\in E$, so the curvature condition is the same whether it is on $a+E$ or, after translation, on $E$. Moreover, if $P_{E}$ is the orthogonal projection onto $E$ and $\Id_{E}$ is the identity on $E$, then $(P_{E})_{\#}\mathcal{N}(0,\Lambda^{-1}\Id)$ is the centred Gaussian on $E$ with covariance $\Lambda^{-1}\Id_{E}$, and the translated target has the same diameter. As in the proof above, the Brenier map is $x\mapsto a+\widetilde T(P_{E}x)$, where $\widetilde T$ is the Brenier map on $E$, and Theorem \ref{thm:sharp-gaussian} applied on $E$ gives \eqref{eq:sharp-gaussian} for $T$.
\end{remark}

\subsection{Convex source potentials}

\begin{lemma}[A bound for the source score]\label{lem:source-score}
Let $\Lambda>0$, let $0\preceq Q\preceq\Lambda\Id$, let $s,a>0$, let $h\geq0$, and let $v\in\R^{d-1}$. In coordinates with first unit vector $e\in\Sd$, we set
\[
\delta:=(h,v),
\qquad
b:=\left(a,-\frac{v}{2s}\right),
\qquad
q:=\langle Qe,e\rangle.
\]
It follows that
\begin{equation}\label{eq:source-score}
\langle b,Q\delta\rangle\leq\frac{q}{8s}(2sa+h)^{2}.
\end{equation}
In particular, if $h<2f\leq2sa$ for some $f\in\R$, then
\[
\langle b,Q\delta\rangle<2\Lambda sa^{2}.
\]
\end{lemma}

\begin{proof}
Let us first suppose that $Q\succ0$, and write
\[
Q=\begin{pmatrix}q&w^{\top}\\w&A\end{pmatrix}.
\]
We complete the square in $v$ and then use the Schur complement estimate $w^{\top}A^{-1}w\leq q$, which gives
\begin{align*}
\langle b,Q\delta\rangle
&=aqh+\left(a-\frac{h}{2s}\right)w\cdot v-\frac1{2s}v^{\top}Av\\
&\leq aqh+\frac{s}{2}\left(a-\frac{h}{2s}\right)^{2}w^{\top}A^{-1}w\\
&\leq q\left[ah+\frac{s}{2}\left(a-\frac{h}{2s}\right)^{2}\right]
=\frac{q}{8s}(2sa+h)^{2}.
\end{align*}
For semidefinite $Q$, it suffices to apply the same calculation to $Q+\eps\Id$ and let $\eps\downarrow0$.

If $q>0$, the inequalities $0\leq h<2f\leq2sa$ make the last expression strictly smaller than \(2qsa^{2}\leq2\Lambda sa^{2}\). If $q=0$, the upper bound in \eqref{eq:source-score} is zero, which is, again, strictly less than $2\Lambda sa^{2}$. In dimension one, there are no transverse terms, and the same estimate follows directly from $(2sa-h)^{2}\geq0$.
\end{proof}

\begin{proof}[Proof of Theorem \ref{thm:convex-source}]
By Lemma \ref{lem:preserve-curvature}, we may first assume that $V$ is smooth with \(0\preceq \mathrm{D}^{2}V\preceq\Lambda\Id\), that $K$ is smooth and strictly convex, and that $W$ is smooth on a neighbourhood of $K$ with $\mathrm{D}^{2}W\succeq-\rho\Id$.

For convex sources, the source score is no longer linear. We thus choose the profile so that it absorbs the worst case allowed by Lemma \ref{lem:source-score}. We fix $R>D/2$, and for $0\leq s<R$, define
\[
u(s):=\Lambda\int_{0}^{R^{2}-s^{2}}e^{\rho t}\dd t,
\qquad
a(s):=u(s)^{-1/2},
\qquad
f(s):=\int_{0}^{s}a(t)\dd t.
\]
Observe that these functions are smooth and $a>0$. Moreover, \(u'(s)=-2\Lambda s e^{\rho(R^{2}-s^{2})} =-2s(\rho u(s)+\Lambda)\), so $a$ is nondecreasing, $0\leq f(s)\leq sa(s)$, and
\begin{equation}\label{eq:convex-profile}
\frac{f''(s)}{f'(s)}=\rho s+\Lambda s f'(s)^{2}\geq0,
\qquad
f'(0)^{-2}=\Lambda\int_{0}^{R^{2}}e^{\rho t}\dd t.
\end{equation}
The sign assertion remains valid for $\rho<0$, because \(\rho u(s)+\Lambda=\Lambda e^{\rho(R^{2}-s^{2})}>0\).
At a negative minimum in Lemma \ref{lem:inverse-reflection}, we define the mean source Hessian along the source segment by \(Q_{*}:=\int_{0}^{1}\mathrm{D}^{2}V(S(x)+t\delta)\dd t\). We have \(0\preceq Q_{*}\preceq\Lambda\Id\), and the fundamental theorem of calculus gives \(\nabla V(S(y))-\nabla V(S(x))=Q_{*}\delta\). Here, \(0\leq\delta_{e}<2f(s)\leq2sa(s)\). We apply Lemma \ref{lem:source-score} with \eqref{eq:convex-profile} to obtain
\[
\frac{2f''(s)}{a(s)}
\leq2\rho s+\langle b,Q_{*}\delta\rangle
<2\rho s+2\Lambda s a(s)^{2}
=\frac{2f''(s)}{a(s)},
\]
which is impossible. Therefore, the inverse comparison is nonnegative, and Lemma \ref{lem:inverse-reflection} gives \(\Lip(T)^{2}\leq f'(0)^{-2} =\Lambda\int_{0}^{R^{2}}e^{\rho t}\dd t\). We now let $R\downarrow D/2$. Since the expression \(\Lambda\int_{0}^{D^{2}/4}e^{\rho t}\dd t\) is increasing in $D$ for every real $\rho$, Lemma \ref{lem:preserve-curvature} passes the estimate to the original data, with the same values of $\Lambda$ and $\rho$. Evaluation of the integral gives
\[
\Lip(T)^{2}\leq
\begin{cases}
\dfrac{\Lambda}{\rho}\bigl(e^{\rho D^{2}/4}-1\bigr),&\rho\neq0,\\[6pt]
\Lambda D^{2}/4,&\rho=0.
\end{cases}
\]
For $\rho>0$, if $r:=\rho D^{2}$, this means that \(\Lip(T)\leq\sqrt\Lambda D \sqrt{\frac{1-e^{-r/4}}{r}}e^{r/8}\).
We must still prove optimality of the exponential rate. The Gaussian sources and quadratic-exponential targets used in the proof of Theorem \ref{thm:sharp-gaussian} belong to the required class. In dimension one, their derivative at zero is
\[
\frac{\sqrt\Lambda}{\sqrt{2\pi}}
\int_{-D/2}^{D/2}e^{\rho t^{2}/2}\dd t
\sim\frac{4\sqrt\Lambda D}{\sqrt{2\pi}\,r}e^{r/8}
\]
by \eqref{eq:B-asymptotic}. For every fixed $d\geq2$, we use the full-dimensional product examples from that proof with \(\eps_{r}:=\frac{D}{r}, \qquad \ell_{r}:=D\sqrt{1-\frac{d-1}{r^{2}}}\) for all sufficiently large $r$. Their diameter is exactly $D$, and \(\rho\ell_{r}^{2}=r-\frac{d-1}{r}\). The same asymptotic with length $\ell_{r}$ shows that the first diagonal derivative at zero is again asymptotic to \(\frac{4\sqrt\Lambda D}{\sqrt{2\pi}\,r}e^{r/8}\). This shows that, in every fixed dimension, no estimate of the form \(C(1+r)^{N}e^{cr}\sqrt\Lambda D\) with fixed finite $C,N$ and $c<1/8$ can hold on this class, which proves sharpness of the exponential rate.
\end{proof}

\section{Sharpness in one dimension}\label{sec:one-dimension}

In this section, we identify the exact one-dimensional constant. On the line, every increasing transport is a Brenier map, so we may factor the transport through the Gaussian and estimate the two factors separately. We write
\[
\varphi(x):=(2\pi)^{-1/2}e^{-x^{2}/2},
\qquad
F_{\gamma}(x):=\int_{-\infty}^{x}\varphi(t)\dd t.
\]
The next lemma gives a source-side estimate directly for finite semiconcave potentials.

\begin{lemma}[Transport from the source to a Gaussian]\label{lem:source-to-gaussian}
Let $\Lambda>0$, and let $V:\R\to\R$ be such that
\[
P(x):=\frac{\Lambda}{2}x^{2}-V(x)
\]
is convex and $e^{-V}$ is integrable, and let $\mu:=Z_{\mu}^{-1}e^{-V}\dd x$. We set
\[
f_{\mu}:=Z_{\mu}^{-1}e^{-V},
\qquad
F_{\mu}(x):=\int_{-\infty}^{x}f_{\mu}(t)\dd t.
\]
It follows that the increasing map
\[
S:=F_{\gamma}^{-1}\circ F_{\mu}
\]
pushes $\mu$ to $\gamma_{1}$ and satisfies
\[
\Lip(S)\leq\sqrt\Lambda.
\]
\end{lemma}

\begin{proof}
The finite convex function $P$ is continuous, so $f_{\mu}$ is positive and continuous. The idea is to compare $\mu$ locally with a Gaussian of precision $\Lambda$, and for this, we just need the semiconcavity of $V$ at a single point. Let us fix $x\in\R$, choose $q\in\partial P(x)$, and set \(a:=\frac{\Lambda x-q}{\sqrt\Lambda}\). By the subgradient inequality, we have, for every $t\in\R$, \(V(x+t)\leq V(x)+\sqrt\Lambda at+\frac{\Lambda}{2}t^{2}\), and so \(f_{\mu}(x+t) \geq f_{\mu}(x)e^{-\sqrt\Lambda at-\Lambda t^{2}/2}\). In other terms, the density of $\mu$ dominates a Gaussian profile which touches it at $x$. Let us set $p:=F_{\mu}(x)$. Integration over the two half-lines gives
\[
p\geq\frac{f_{\mu}(x)}{\sqrt\Lambda}\frac{F_{\gamma}(a)}{\varphi(a)},
\qquad
1-p\geq\frac{f_{\mu}(x)}{\sqrt\Lambda}
\frac{1-F_{\gamma}(a)}{\varphi(a)}.
\]
Indeed,
\[
\frac{F_{\gamma}(a)}{\varphi(a)}
=\int_{0}^{\infty}e^{at-t^{2}/2}\dd t,
\qquad
\frac{1-F_{\gamma}(a)}{\varphi(a)}
=\int_{0}^{\infty}e^{-at-t^{2}/2}\dd t.
\]
The ratio $\varphi/F_{\gamma}$ is decreasing, while $\varphi/(1-F_{\gamma})$ is increasing. Let \(z:=F_{\gamma}^{-1}(p)\). If $a\geq z$, the first inequality gives
\[
\frac{f_{\mu}(x)}{\sqrt\Lambda}
\leq p\frac{\varphi(a)}{F_{\gamma}(a)}
\leq p\frac{\varphi(z)}{F_{\gamma}(z)}
=\varphi(z).
\]
If $a\leq z$, the second inequality and the monotonicity of the Gaussian hazard give the same conclusion. We obtain \(f_{\mu}(x)\leq\sqrt\Lambda\,\varphi(S(x)) \qquad (x\in\R)\). Finally, the map $S$ is $C^{1}$, and \(S'(x)=\frac{f_{\mu}(x)}{\varphi(S(x))}\leq\sqrt\Lambda\), which proves the claim.
\end{proof}

\begin{proof}[Proof of Theorem \ref{thm:one-dimensional}]
Let $S$ be the increasing map from $\mu$ to $\gamma_{1}$ given by Lemma \ref{lem:source-to-gaussian}, and let $R$ be the increasing transport from $\gamma_{1}$ to $\nu$. By Theorem \ref{thm:sharp-gaussian}, we have \(\Lip(R)\leq\frac1{\sqrt{2\pi}} \int_{-D/2}^{D/2}e^{\rho t^{2}/2}\dd t\). The composition $R\circ S$ is increasing and pushes $\mu$ to $\nu$, so it is the one-dimensional Brenier map. If we combine the two Lipschitz estimates, we obtain \eqref{eq:one-dimensional}, and the choice $\rho=0$ gives \eqref{eq:one-dimensional-logconcave}.

Sharpness follows from the one-dimensional examples in the proof of Theorem \ref{thm:sharp-gaussian}. After a translation of the target, they show that the Gaussian source of variance $\Lambda^{-1}$ and the target density proportional to $e^{\rho(y-(a+b)/2)^{2}/2}$ on $(a,b)$ give equality in \eqref{eq:one-dimensional}. Finally, the asymptotic \eqref{eq:B-asymptotic} proves the optimality of the exponential rate $1/8$, which includes the polynomial prefactor.
\end{proof}

For the standard Gaussian source and log-concave targets, we can also identify the pointwise equality case in \eqref{eq:one-dimensional-logconcave}. Let $\nu$ be a log-concave probability measure whose support is an interval of length $D$, with density $f_{\nu}$ and distribution function $F_{\nu}$, and let $T$ be the increasing map from $\gamma_{1}$ to $\nu$. Differentiation of the identity $F_{\nu}(T(x))=F_{\gamma}(x)$ gives
\[
T'(x)=\frac{\varphi(x)}{f_{\nu}(T(x))}
=\frac{I_{\gamma_{1}}(p)}{I_{\nu}(p)},
\qquad p:=F_{\gamma}(x),
\]
where $I_{\nu}:=f_{\nu}\circ F_{\nu}^{-1}$ and $I_{\gamma_{1}}:=\varphi\circ F_{\gamma}^{-1}$ are the quantile profiles defined in Proposition \ref{prop:target-profile-equality} below. In other terms, the inequality in Proposition \ref{prop:target-profile-equality} is the pointwise form of \eqref{eq:one-dimensional-logconcave} with $\Lambda=1$, that is, $T'(x)\leq D/\sqrt{2\pi}$ for every $x$. This inequality follows from an isoperimetric inequality of Montenegro \cite[(2) and (9)]{Montenegro2006}, which we apply to half-lines (see Section \ref{sec:previous-work}). The new part is the equality case, namely, $T'(x)=D/\sqrt{2\pi}$ holds at some point $x$ if and only if $x=0$ and $\nu$ is uniform on its support. Our argument is based on the concavity of the quantile profile and on an elementary inequality for the Gaussian Mills ratio, which we prove first.

\begin{lemma}[A Gaussian Mills ratio inequality]\label{lem:mills-ratio}
Let $c:=\sqrt{2\pi}$. For every $x\geq0$,
\begin{equation}\label{eq:mills-ratio}
\varphi(x)-x\bigl(1-F_{\gamma}(x)\bigr)\leq\frac{x}{e^{cx}-1},
\end{equation}
where the right-hand side is understood as $1/c$ at $x=0$. Moreover, equality holds only at $x=0$.
\end{lemma}

\begin{proof}
Both sides of \eqref{eq:mills-ratio} equal $1/c$ at $x=0$, so it suffices to compare their logarithmic derivatives on $(0,\infty)$. This comparison reduces to an inequality between the Gaussian Mills ratio and a Laplace transform. Let us set
\[
m(x):=\frac{1-F_{\gamma}(x)}{\varphi(x)}=\int_{0}^{\infty}e^{-xt-t^{2}/2}\dd t,
\qquad
q(x):=\int_{0}^{c}\left(1-\frac tc\right)e^{-xt}\dd t.
\]
In other words, $m$ and $q$ are the Laplace transforms of the Gaussian kernel $e^{-t^{2}/2}$ and of the tent kernel $\max\{1-t/c,0\}$ on $[0,\infty)$. For $x>0$, direct integration gives
\[
q(x)=\frac1x-\frac{1-e^{-cx}}{cx^{2}}.
\]

At the beginning, we must compare the kernels rigorously. We claim that $m(x)>q(x)$ for each $x>0$. To this end, let us consider the difference
\[
k(t):=e^{-t^{2}/2}-\max\left\{1-\frac tc,0\right\}
\qquad (t\geq0).
\]
Since $\int_{0}^{\infty}e^{-t^{2}/2}\dd t=c/2$ and $\int_{0}^{\infty}te^{-t^{2}/2}\dd t=1$, we have
\[
\int_{0}^{\infty}k(t)\dd t=0,
\qquad
\int_{0}^{\infty}tk(t)\dd t=1-\frac{\pi}{3}<0.
\]
This means that the two kernels have the same mass, but the tent kernel has the larger first moment. We now show that $k$ changes sign exactly twice, with the sign pattern $+,-,+$. For $t\geq c$, we have $k(t)=e^{-t^{2}/2}>0$. On $(0,c)$, the sign of $k$ is the sign of
\[
\psi(t):=-\frac{t^{2}}{2}-\log\left(1-\frac tc\right),
\qquad
\psi'(t)=\frac1{c-t}-t.
\]
The zeros of $\psi'$ are the roots of $t^{2}-ct+1=0$. Given that $c^{2}=2\pi>4$, there exist exactly two of them, $t_{-}<t_{+}$, and since $t_{-}t_{+}=1$ and $t_{-}+t_{+}=c$, both lie in $(0,c)$ and $t_{-}<1$. It follows that $\psi$ increases on $(0,t_{-})$, decreases on $(t_{-},t_{+})$, and increases on $(t_{+},c)$. Note also that $\psi(0)=0$, $\psi(t)\to+\infty$ as $t\uparrow c$, and
\[
\psi(\sqrt2)=-1-\log\left(1-\frac1{\sqrt\pi}\right)<0.
\]
Indeed, $\pi>3$ gives $1-1/\sqrt\pi>1-1/\sqrt3>3/8$, while $e>8/3$ gives $e^{-1}<3/8$. Since $t_{-}<\sqrt2<c$ and $t_{+}$ is the minimum point of $\psi$ on $[t_{-},c)$, we also have $\psi(t_{+})<0$. This implies that $\psi$ is positive on $(0,t_{-}]$ and has exactly one zero $\alpha\in(t_{-},t_{+})$ and exactly one zero $\beta\in(t_{+},c)$. We conclude that $k>0$ on $(0,\alpha)\cup(\beta,\infty)$ and $k<0$ on $(\alpha,\beta)$.

Let us fix $x>0$, and let $\ell(t)=\ell_{0}+\ell_{1}t$ be the affine function which agrees with the strictly convex function $t\mapsto e^{-xt}$ at $\alpha$ and $\beta$. The function $t\mapsto e^{-xt}$ is decreasing, so $\ell_{1}<0$. By strict convexity, $e^{-xt}-\ell(t)$ is positive outside $[\alpha,\beta]$ and negative on $(\alpha,\beta)$, so it has the same sign as $k$ away from $\alpha$ and $\beta$. This explains the choice of $\ell$, and we deduce that
\begin{align*}
m(x)-q(x)
&=\int_{0}^{\infty}k(t)e^{-xt}\dd t
>\int_{0}^{\infty}k(t)\ell(t)\dd t\\
&=\ell_{0}\int_{0}^{\infty}k(t)\dd t+\ell_{1}\int_{0}^{\infty}tk(t)\dd t
=\ell_{1}\left(1-\frac{\pi}{3}\right)>0.
\end{align*}
This proves the claim.

We now consider the logarithmic derivative. Let us set
\[
H(x):=\varphi(x)-x\bigl(1-F_{\gamma}(x)\bigr)=\int_{x}^{\infty}(t-x)\varphi(t)\dd t,
\qquad
\omega(x):=\frac{x}{e^{cx}-1},
\]
with $\omega(0):=1/c$. Both functions are positive and continuous on $[0,\infty)$, and $H(0)=\varphi(0)=1/c=\omega(0)$. Since $\varphi'(x)=-x\varphi(x)$, we have $H'=-(1-F_{\gamma})$ and $H=\varphi(1-xm)$. For $x>0$, it follows that
\[
\frac{\dd}{\dd x}\log\frac{H(x)}{\omega(x)}
=-\frac{m(x)}{1-xm(x)}-\frac1x+\frac{c}{1-e^{-cx}}.
\]
The right-hand side is negative if and only if $m(x)>q(x)$. Indeed, $xm(x)<1$ because $H(x)>0$, the map $z\mapsto z/(1-xz)$ is increasing on $[0,1/x)$, and a direct computation shows that it takes the value $\frac{c}{1-e^{-cx}}-\frac1x$ at $z=q(x)$. By the first step, the function $\log(H/\omega)$ is strictly decreasing on $(0,\infty)$. Because it is continuous on $[0,\infty)$ and vanishes at $x=0$, we conclude that $H<\omega$ on $(0,\infty)$, which is \eqref{eq:mills-ratio} with strict inequality for $x>0$.
\end{proof}

\begin{proposition}[Equality in the target profile comparison]\label{prop:target-profile-equality}
Let $\nu$ be a log-concave probability measure whose support is a nondegenerate compact interval of length $D$, and let $f_{\nu}$ be its continuous log-concave density on the interior of the interval, with distribution function $F_{\nu}$. We define
\[
I_{\nu}(p):=f_{\nu}(F_{\nu}^{-1}(p)),
\qquad
I_{\gamma_{1}}(p):=\varphi(F_{\gamma}^{-1}(p)).
\]
It follows that
\[
I_{\nu}(p)\geq\frac{\sqrt{2\pi}}{D}I_{\gamma_{1}}(p)
\qquad (0<p<1).
\]
Moreover, equality holds at a quantile if and only if $p=1/2$ and $\nu$ is uniform on its support.
\end{proposition}

\begin{proof}
By scaling, it is enough to take $D=1$. In the proof, we replace the concave profile $I_{\nu}$ by a tangent line, which reduces the comparison to an explicit family of affine profiles. We then compare these affine profiles with the Gaussian one by Lemma \ref{lem:mills-ratio}.

\textit{The profile $I_{\nu}$.}
The concavity of $I_{\nu}$ is a classical observation due to Bobkov \cite[Proposition A.1]{Bobkov1996}, whose argument we recall here. The function $\log f_{\nu}$ is finite and concave on the interior of the support, so $f_{\nu}$ is positive and locally Lipschitz there. Since $F_{\nu}^{-1}$ is $C^{1}$ on $(0,1)$, the profile $I_{\nu}$ is locally absolutely continuous on $(0,1)$, and
\[
I_{\nu}'(p)=(\log f_{\nu})'(F_{\nu}^{-1}(p))
\]
for almost every $p\in(0,1)$. The right-hand side is nonincreasing, which implies that $I_{\nu}$ is concave. The substitution $q=F_{\nu}(y)$ also shows that \(\int_{0}^{1}\frac{\dd q}{I_{\nu}(q)}=1\).

\textit{The supporting line.}
We fix $p\in(0,1)$, and let $L$ be a supporting affine function at $p$, so that $L(p)=I_{\nu}(p)$. Such a function is positive on $[0,1]$ and dominates $I_{\nu}$. Indeed, $L\geq I_{\nu}>0$ on $(0,1)$ by concavity, so $L(0),L(1)\geq0$ by continuity. If $L(0)=0$, then $L$ is a positive multiple of $q\mapsto q$, and
\[
1=\int_{0}^{1}\frac{\dd q}{I_{\nu}(q)}\geq\int_{0}^{1}\frac{\dd q}{L(q)}=+\infty,
\]
which is impossible. The case $L(1)=0$ is analogous. We set \(\theta:=\int_{0}^{1}\frac{\dd q}{L(q)}\leq1\), and observe that the affine function $\widetilde L:=\theta L$ has reciprocal integral one and satisfies \(\widetilde L(p)\leq I_{\nu}(p)\).

\textit{Affine profiles.}
Every positive affine function with reciprocal integral one has the form \(L_{\lambda}(q):=\frac{\lambda}{e^{\lambda}-1}+\lambda q\), with $L_{0}\equiv1$. Indeed, if $a+\lambda q$ is positive on $[0,1]$ and $\lambda\neq0$, its reciprocal integral is $\lambda^{-1}\log(1+\lambda/a)$, which equals one exactly when $a=\lambda/(e^{\lambda}-1)$. The same family appears in the proof of Theorem 2.4 of Montenegro \cite[(7)]{Montenegro2006}. The strict Gaussian comparison for these affine profiles states that \(L_{\lambda}(q)\geq\sqrt{2\pi}\,I_{\gamma_{1}}(q)\), with equality only when $\lambda=0$ and $q=1/2$. We now prove this comparison. Let us write
\[
\Gamma(q):=\sqrt{2\pi}\,I_{\gamma_{1}}(q)=e^{-F_{\gamma}^{-1}(q)^{2}/2}
\qquad (0<q<1).
\]
Since $L_{\lambda}(1-q)=L_{-\lambda}(q)$ and $\Gamma(1-q)=\Gamma(q)$, it suffices to consider $\lambda\geq0$. Observe that $\Gamma$ is strictly concave because $\Gamma'(q)=-\sqrt{2\pi}\,F_{\gamma}^{-1}(q)$ decreases from $+\infty$ to $-\infty$. It follows that $L_{\lambda}-\Gamma$ is strictly convex and attains its minimum at a unique point $q_{\lambda}\in(0,1)$, characterised by $\Gamma'(q_{\lambda})=\lambda$. If we write $\lambda=\sqrt{2\pi}\,x$ with $x\geq0$, this means that $F_{\gamma}^{-1}(q_{\lambda})=-x$, that is, $q_{\lambda}=1-F_{\gamma}(x)$ and $\Gamma(q_{\lambda})=\sqrt{2\pi}\,\varphi(x)$. Therefore, with the convention of Lemma \ref{lem:mills-ratio} at $x=0$,
\[
L_{\lambda}(q_{\lambda})-\Gamma(q_{\lambda})
=\sqrt{2\pi}\left[\frac{x}{e^{\sqrt{2\pi}x}-1}+x\bigl(1-F_{\gamma}(x)\bigr)-\varphi(x)\right]
\geq0
\]
by Lemma \ref{lem:mills-ratio}, with equality only when $x=0$. We conclude that $L_{\lambda}\geq\Gamma$ on $(0,1)$, and that equality at some $q$ forces $q=q_{\lambda}$ and $x=0$, so $\lambda=0$ and $q=1/2$. With \(\widetilde L(p)\leq I_{\nu}(p)\), this gives the required inequality.

\textit{The equality case.}
Suppose now that equality holds at $p$. In this case, equality holds both in $\widetilde L(p)\leq I_{\nu}(p)$ and in $\widetilde L(p)\geq\Gamma(p)$. The second equality forces $p=1/2$ and $\widetilde L=L_{0}\equiv1$. Since $I_{\nu}(p)>0$, the first one gives $\theta=1$. It follows that $L\equiv1$. Given that $L\geq I_{\nu}$ and the reciprocal integrals agree, we obtain $I_{\nu}=L$ almost everywhere, and hence everywhere by continuity. This means that $I_{\nu}\equiv1$, and hence $\nu$ is uniform. Conversely, the uniform measure gives equality at $p=1/2$.
\end{proof}

\begin{remark}[Scope of composition]
The identity $T=R\circ S$ in the proof of Theorem \ref{thm:one-dimensional} relies on the uniqueness of the increasing transport on the line. In higher dimensions, however, a composition of two gradients of convex functions \textit{need not} be a gradient of a convex function. For this reason, this argument does not extend the sharp Gaussian estimate to the full source class of Theorem \ref{thm:support-curvature}.
\end{remark}

\section{Optimality and open questions}\label{sec:optimality}

\begin{question}\label{q:open}
Do the constants of Section \ref{sec:constants} satisfy $C_{*}^{G}=C_{*}^{\mathrm{full}}=1/\sqrt{2\pi}$ and $c_{\mathrm{sl}}^{*}=1/8$? We conjecture this is the case.
\end{question}

\begin{remark}[Optimality within the scalar criterion]\label{rem:scalar-optimality}
The constant $C_{1}$ is optimal within the scalar admissibility criterion of Proposition \ref{prop:profile}. Indeed, let us normalise $\Lambda=w=1$ and set $p:=\Theta'$ for a penalty $\Theta$ which satisfies \eqref{eq:profile-condition}. In this case, $p(0)=0$, $p'>0$, $p(s)\to\infty$, and \(\frac{p'}{p}<\frac1{s\Qp}=\frac{\Sp'}{\Sp}\) whenever $0<p\leq1$. If $s_{*}$ is the first point with $p(s_{*})=1$, the function $p/\Sp$ is strictly decreasing on $(0,s_{*}]$ and tends to $p'(0)$ as $s\downarrow0$. It follows that \(1=p(s_{*})\leq p'(0)\Sp(s_{*})<p'(0)\Sp_{\infty}\), that is, $\Theta''(0)>1/\Sp_{\infty}=C_{1}$. On the other hand, the construction in Proposition \ref{prop:C1} approaches $C_{1}$. We emphasise that this does not exclude stronger translated-corrector arguments which use additional information about the transport.
\end{remark}

\section*{AI disclosure}

Codex was used to set up custom scripts for numerical integration and to provide more empirical evidence that points towards the conjectures listed as the open problems. It was also used to check for possible typos and verify whether the notation has been properly introduced and that the symbols were unique.

\bibliographystyle{amsplain}
\bibliography{dimfree_v2}

\end{document}